\documentclass[a4paper,11pt,french,english]{amsart}
\usepackage[english]{babel}
\usepackage[utf8]{inputenc}
\usepackage{lmodern}
\usepackage[T1]{fontenc}
\usepackage{amsmath,amsthm,amssymb,amsfonts}
\usepackage{amscd}
\usepackage{pstricks}
\usepackage{pst-node}
\usepackage[a4paper,left=3.4cm,right=3.4cm,top=4cm,bottom=4cm]{geometry}
\usepackage{enumerate}
\usepackage{enumitem}
\usepackage[linktocpage=true]{hyperref}
\usepackage{url}
\usepackage{csquotes}

\newcommand{\Sy}{\mathrm{Sym}(\Omega)}

\newcommand{\aut}{\mathrm{Aut}(T_d)}
\newcommand{\Aut}{\mathrm{Aut}}

\newcommand{\sub}{\ensuremath{\operatorname{Sub}}}%
\newcommand{\urs}{\ensuremath{\operatorname{URS}}}
\newcommand{\fix}{\ensuremath{\mathrm{Fix}}}%
\renewcommand{\H}{\ensuremath{\mathcal{H}}}
\newcommand{\K}{\ensuremath{\mathcal{K}}}%
\newcommand{\A}{\ensuremath{\mathcal{A}}}%
\newcommand{\X}{\ensuremath{\mathcal{X}}}%
\newcommand{\mon}{\ensuremath{\mathrm{Mon}}}%
\newcommand{\env}{\ensuremath{\mathrm{Env}}}%
\newcommand{\envA}{\ensuremath{\mathrm{Env}(\mathcal{A}_\Gamma)}}%
\newcommand{\prob}{\ensuremath{\mathrm{Prob}}}%

\title[Amenable URS's and lattice embeddings]{Amenable uniformly recurrent subgroups and lattice embeddings}
\author{Adrien Le Boudec}
\thanks{This work was carried out when the author was F.R.S.-FNRS Postdoctoral Researcher. Current affiliation: CNRS, UMPA - ENS Lyon. Partially supported by ANR-14-CE25-0004 GAMME}
\address{UCLouvain, IRMP,	Chemin du Cyclotron 2, 1348 Louvain-la-Neuve, Belgium}
\address{CNRS, Unité de Mathématiques Pures et Appliquées,	ENS-Lyon, France}
\email{adrien.le-boudec@ens-lyon.fr}
\date{January 23, 2020}

\theoremstyle{plain}
\newtheorem{thm}{Theorem}[section]
\newtheorem{prop}[thm]{Proposition}
\newtheorem{cor}[thm]{Corollary}
\newtheorem{lem}[thm]{Lemma}

\newtheorem*{thm-intro}{Theorem}

\theoremstyle{definition}
\newtheorem{defi}[thm]{Definition}
\newtheorem{quest}[thm]{Question}
\newtheorem{ex}[thm]{Example}
\newtheorem{rmq}[thm]{Remark}

\begin{document}

\maketitle

\begin{abstract}
We study lattice embeddings for the class of countable groups $\Gamma$ defined by the property that the largest amenable uniformly recurrent subgroup $\A_\Gamma$ is continuous. When $\A_\Gamma$ comes from an extremely proximal action and the envelope of $\A_\Gamma$ is co-amenable in $\Gamma$, we obtain restrictions on the locally compact groups $G$ that contain a copy of $\Gamma$ as a lattice, notably regarding normal subgroups of $G$, product decompositions of $G$, and more generally dense mappings from $G$ to a product of locally compact groups.

\end{abstract}



\section{Introduction}

The questions considered in this article fall into the setting of the following general problem: given a (class of) countable group $\Gamma$, study the locally compact groups $G$ such that $\Gamma$ embeds as a lattice in $G$, i.e.\ such that $\Gamma$ sits as a discrete subgroup of $G$ and $G / \Gamma$ carries a $G$-invariant probability measure.

Malcev showed that every finitely generated torsion free nilpotent group embeds as a cocompact lattice in a unique simply connected nilpotent Lie group \cite[Ch.\ II]{Raghu-dis-sb}. Conversely if $G$ is a locally compact group with a finitely generated nilpotent lattice, then after modding out by a compact normal subgroup, the identity component $G^0$ is a Lie group of polynomial growth (these have been characterized in \cite{Guiv-LiePol,Jen-LiePol}) and $G / G^0$ is finitely generated and virtually nilpotent. This statement is a combination of several works. First if $G$ has a finitely generated nilpotent lattice $\Gamma$, then $\Gamma$ is necessarily cocompact in $G$. Since $\Gamma$ is virtually torsion free this is a classical fact when $G$ is totally disconnected, and the general case can be deduced from \cite[Prop.\ 3.7]{BenQui-Sadic} (which uses notably the solution of Hilbert's fifth problem \cite{MontZip}). In particular $G$ is compactly generated with polynomial growth, and the statement then follows from the generalization of Gromov's polynomial growth theorem for locally compact groups \cite{Los-PG}.

Beyond the nilpotent case, examples of classifications of embeddings of $\Gamma$ as a cocompact lattice have been obtained by Dymarz in \cite{Dym-env} for several families of examples of solvable groups $\Gamma$. Although not directly related to our concerns, we also mention that a certain dual problem was considered by Bader--Caprace--Gelander--Mozes in \cite{BCGM-lat-am} for the class of amenable groups.

Outside the setting of amenable groups, Furman addressed the above problem for the class of lattices $\Gamma$ in semi-simple Lie groups in \cite{Fur-MMTarget}, improving rigidity results of Mostow, Prasad, Margulis (see the references in \cite{Fur-MMTarget}; see also Furstenberg \cite{Furst-poiss-env}). In \cite{BFS-env}, Bader--Furman--Sauer considered a large class of countable groups $\Gamma$ defined by certain group theoretic conditions, and established, given a lattice embedding of $\Gamma$ in $G$, a general arithmeticity result in the setting where the connected component of $G$ is non-compact.

In this article we consider the class of groups whose Furstenberg uniformly recurrent subgroup is continuous, and we address the question of how the properties of the Furstenberg uniformly recurrent subgroup of a countable group $\Gamma$ influence the locally compact groups into which $\Gamma$ embeds as a lattice. Below we explain the above terminology and include a discussion on this class of groups. Let us also mention at this point that examples of lattice embeddings for groups within this class are described in \cite{LB-irr-wreath} (see Remark \ref{rmq-short}). We encourage the reader to read \cite{LB-irr-wreath}, that may be considered as a companion article of the present work, in the sense that the examples from \cite{LB-irr-wreath} may be viewed as motivating examples for the problems that we address here.

\subsection*{The groups under consideration}

For a countable group $\Gamma$, the Chabauty space $\sub(\Gamma)$ of all subgroups of $\Gamma$ is a compact space, on which $\Gamma$ acts by conjugation. A uniformly recurrent subgroup (URS) of $\Gamma$ is a closed minimal $\Gamma$-invariant subset of $\sub(\Gamma)$ \cite{Gla-Wei}. Glasner and Weiss showed that every minimal action of $\Gamma$ on a compact space $X$ gives rise to a URS (see Proposition \ref{prop-GW-1.2}), called the stabilizer URS associated to the action. Conversely every URS arises as the stabilizer URS of a minimal action (see Matte Bon--Tsankov \cite{MB-Ts-realizing}, and Elek \cite{Elek-urs} in the case of finitely generated groups).

URS's have been shown to be related to the study of ideals in reduced group $C^\ast$-algebras \cite{KK,Kenn} and reduced crossed products \cite{kawabe-urs}. URS's of several classes of groups have been studied in \cite{LBMB}. For certain examples of groups $\Gamma$, rigidity results about minimal actions on compact spaces have been obtained in \cite{LBMB} from a complete description of the space $\urs(\Gamma)$. Various results about homomorphisms between topological full groups of étale groupoids, notably obstructions involving invariants of the groupoids, have been obtained in \cite{MB-urs-full} via URS's considerations (more precisely via a complete description of the points in the Chabauty space of these groups whose orbit does not approach the trivial subgroup). In the present article we will make use of URS's as a tool in order to study lattice embeddings for a class of countable groups that we now define.

A URS is amenable if it consists of amenable subgroups. Every countable group $\Gamma$ admits a largest amenable URS $\A_\Gamma$ (with respect to a natural partial order on $\urs(\Gamma)$, see \S \ref{subsec-gen-urs}), which is the stabilizer URS associated to the action of $\Gamma$ on its Furstenberg boundary (see \S \ref{subsec-top-bnd} for definitions). The URS $\A_\Gamma$ is called the Furstenberg URS of $\Gamma$. $\A_\Gamma$ is either a point, in which case we have $\A_\Gamma = \left\{\mathrm{Rad}(\Gamma)\right\}$, where $\mathrm{Rad}(\Gamma)$ is the amenable radical of $\Gamma$, or homeomorphic to a Cantor space. In this last case we say that $\A_\Gamma$ is \textbf{continuous}. We refer to \cite{LBMB} for a more detailed discussion.

Let $(\mathcal{C})$ denote the class of groups $\Gamma$ for which the Furstenberg URS $\A_\Gamma$ is continuous. Equivalently, a group $\Gamma$ belongs to $(\mathcal{C})$ if and only if $\Gamma$ admits an amenable URS whose envelope is not amenable (see below for the definition of the envelope). The class $(\mathcal{C})$ is disjoint from all classes of groups previously mentioned in the introduction. More precisely, the class $(\mathcal{C})$ is disjoint from the class of amenable groups, the class of linear groups \cite{BKKO}, and also from other classes of groups specifically considered in \cite{BFS-env}, such as groups with non-vanishing $\ell^2$-Betti numbers \cite{BKKO} or acylindrically hyperbolic groups (see \cite[Th.\ 7.19]{DGO} and \cite[Th.\ 1.4]{BKKO}). Examples of groups in $(\mathcal{C})$ include groups acting on trees with all end stabilizers amenable and non-trivial (see the discussion after Corollary \ref{cor-intro-tree-disc} for explicit examples). The class $(\mathcal{C})$ is stable under taking quotient by an amenable normal subgroup and extension by an amenable group \cite[Prop.\ 2.20]{LBMB}. Also if $\Gamma$ has a normal subgroup that is in $(\mathcal{C})$, then $\Gamma$ belongs to $(\mathcal{C})$ \cite[Prop.\ 2.24]{LBMB}. By a result of Breuillard--Kalantar--Kennedy--Ozawa, the complement of the class $(\mathcal{C})$ is also stable under extensions (see \cite[Prop.\ 2.24]{LBMB}).

The study of this class of groups is also motivated by the work of Kalantar--Kennedy \cite{KK}, who showed the following characterization: a countable group $\Gamma$ belongs to $(\mathcal{C})$ if and only if the group $\Gamma / \mathrm{Rad}(\Gamma)$ has a reduced $C^\ast$-algebra that is not simple. For an introduction and the historical developments of the problem of $C^\ast$-simplicity, we refer to the survey of de la Harpe \cite{dlH-survey}.

\subsection*{Topological boundaries}

We will make use of the notion of topological boundary in the sense of Furstenberg. These are compact spaces with a minimal and strongly proximal group action (see \S \ref{subsec-top-bnd} for definitions). Many different notions of boundaries appear in the study of groups and group actions. What is now sometimes called \enquote{boundary theory} is particularly well described in the introduction of \cite{BF-icm}. We insist that in the present article the term \textit{boundary} will always refer to a topological boundary in the above sense. This notion should not be confused with any of the measured notions of boundaries. In particular, despite the possibly confusing terminology, the maximal topological boundary, called the Furstenberg boundary, is not the same notion as the measured notion of Poisson--Furstenberg boundary (for which we refer to \cite{Erschler-PF} for a recent survey).

\subsection*{Lattices and direct products}

Special attention will be given to products of locally compact groups. The study of lattices in product groups is motivated (among other things) by its connections with the theory of lattices in semi-simple Lie groups, its rich geometric aspects, as well as the instances of groups with rare properties appearing in this setting. We refer to the literature (see \cite{Margulis-book,Wise-phd,Bur-Moz-cras,BM-IHES-2,Remy-KM-cras,Sha-commens-inv,Bur-Mon-bound-coho-rigid,Ratta-phd,Mon-Sha-cocy,Bad-Sha-NST,Ca-Re-inv,CaMo-KM,Radu-prod-trees}) for developments over the last years on the study of lattices in products of locally compact groups.

Given a countable group $\Gamma$ with a continuous Furstenberg URS and a group $G$ containing $\Gamma$ as a lattice, we are interested in understanding how close the group $G$ can be from a direct product of two groups, or which properties the group $G$ can share with a direct product. Of course various notions of closeness can be considered. The most basic one is to ask whether the group $G$ admits non-trivial decompositions as a direct product. One step further, one might consider quotient morphisms from $G$ onto direct products of groups. In Theorems \ref{thm-intro-URS-lattice-prod} and \ref{thmintro-normal-coamenable-lc} below we more generally consider continuous morphisms with dense image from $G$ to a direct product of groups $G \rightarrow G_1 \times G_2$. We make no assumption about injectivity of these maps or injectivity of the composition with the projection to one factor $G_i$. In particular this setting allows maps of the form $G \rightarrow G/N_1 \times G/N_2$ for closed normal subgroups $N_1,N_2$ such that $N_1 N_2$ is dense in $G$.

\subsection*{Results}

A central notion in this article is the one of extremely proximal action. Minimal and extremely proximal actions naturally arise in geometric group theory, and are boundaries in the sense of Furstenberg. We refer to \S \ref{subsec-EP} for definitions and examples. We say that the Furstenberg URS $\mathcal{A}_\Gamma$ of a countable group $\Gamma$ comes from an extremely proximal action if there exists a compact space $Z$ and a $\Gamma$-action on $Z$ that is minimal and extremely proximal, whose associated stabilizer URS is equal to $\mathcal{A}_\Gamma$. Note that typically $Z$ will not be the Furstenberg boundary of $\Gamma$. If $\H$ is a URS of $\Gamma$, the \textbf{envelope} $\env(\H)$ of $\H$ is by definition the subgroup of $\Gamma$ generated by all the subgroups $H \in \H$.

\begin{thm} \label{thm-intro-URS-lattice-prod}
Let $\Gamma$ be a countable group whose Furstenberg URS comes from a faithful and extremely proximal action, and let $G$ be a locally compact group containing $\Gamma$ as a lattice. The following hold:
\begin{enumerate}[label=(\alph*)]
\item \label{item-intro-env-tf} Assume that $\envA$ is finitely generated and co-amenable in $\Gamma$. Then $G$ cannot be a direct product $G = G_1 \times G_2$ of two non-compact groups.
\item \label{item-intro-env-f-i} Assume that $\envA$ has finite index in $\Gamma$ and finite abelianization. Then any continuous morphism with dense image from $G$ to a product of locally compact groups $G \rightarrow G_1 \times G_2$ is such that one factor $G_i$ is compact.
\end{enumerate}
\end{thm}

This result has applications to the setting of groups acting on trees, see Corollary \ref{cor-intro-tree-disc}. We make several comments about the theorem:

\begin{enumerate}[itemsep=1.7ex, leftmargin = 0.7cm, label=\arabic*)]

\item We do not assume that $\Gamma$ is finitely generated, nor that $G$ is compactly generated. For statement \ref{item-intro-env-tf}, The assumption that $\envA$ if finitely generated admits variations, see Theorem \ref{thm-URS-lattice-prod}.

\item Making an assumption on the \enquote{size} of the envelope of $\A_\Gamma$ with respect to $\Gamma$ is natural, in the sense that in general there is no hope to derive any conclusion on the entire group $\Gamma$ if this envelope is too small. An extreme illustration of this is that there are groups $\Gamma$ whose Furstenberg URS comes from a faithful and extremely proximal action but is trivial, and these can be lattices in products, e.g.\ $\mathrm{PSL}(2,\mathbb{Z}[1/p])$ inside $\mathrm{PSL}(2,\mathbb{R}) \times \mathrm{PSL}(2,\mathbb{Q}_p)$ (see also the discussion right after Corollary \ref{cor-intro-tree-disc}).

\item Under the assumption that $\envA$ is co-amenable in $\Gamma$, the fact that the Furstenberg URS $\A_\Gamma$ comes from a faithful and extremely proximal action is equivalent to asking that the action of $\Gamma$ on $\A_\Gamma$ is faithful and extremely proximal; see Remark \ref{rmq-reform-comes-from-EP}. This provides an intrinsic reformulation of the assumption not appealing to any auxiliary space.

\item For $\Gamma$ as in the theorem, the assumption in statement \ref{item-intro-env-f-i} that $\envA$ has finite index in $\Gamma$ and $\envA$ has finite abelianization is equivalent to $\Gamma$ being virtually simple (see Proposition \ref{prop-core-deriv0}).

\end{enumerate}


\medskip

The URS approach to study lattice embeddings allows to consider more generally subgroups of finite covolume. Recall that a closed subgroup $H$ of a locally compact group $G$ has \textbf{finite covolume} in $G$ if $G/H$ carries a $G$-invariant probability measure. Thus a lattice is a discrete subgroup of finite covolume. Before stating the following result we need some terminology.

Recall the notion of disjointness introduced by Furstenberg in \cite{Furst-disjoint}. If $X,Y$ are compact $G$-spaces, $X$ and $Y$ are disjoint if whenever $\Omega$ is a compact $G$-space and $\Omega \rightarrow X$ and $\Omega \rightarrow Y$ are continuous equivariant surjective maps, the map $\Omega \rightarrow X \times Y$ that makes the natural diagram commute remains surjective (see \S \ref{subsec-presc-urs}). When $X,Y$ are minimal $G$-spaces, this is equivalent to asking that the diagonal $G$-action on the product $X \times Y$ is minimal.

Consider the following property: \textit{two non-trivial $G$-boundaries are never disjoint}. A group with this property will be called \textbf{boundary indivisible}. Glasner characterized minimal compact $G$-spaces which are disjoint from all $G$-boundaries as those carrying a fully supported measure whose orbit closure in the space of probability measures is minimal \cite[Th.\ 6.2]{Glas-compress}. The relation between disjointness and boundaries that we consider here is of different spirit, as it deals with disjointness within the class of $G$-boundaries, rather than disjointness from this class. Locally compact groups with a cocompact amenable maximal subgroup are examples of boundary indivisible groups \cite[Prop.\ 4.4]{Furst-bd-th}. On the contrary, many discrete groups are not boundary indivisible. The relevance of this property in our setting comes from the fact that, as we will show in Proposition \ref{prop-A_G-EP}, a discrete group $\Gamma$ as in Theorem \ref{thm-intro-URS-lattice-prod} is boundary indivisible. Actually the only examples of (non-amenable) boundary indivisible discrete groups that we are aware of fall into the setting of Proposition \ref{prop-A_G-EP}.

Recall that a convex compact $G$-space is irreducible if it does not contain any proper closed convex $G$-invariant subspace. We say that a subgroup $L$ of a topological group $G$ is \textbf{weakly co-amenable} in $G$ if whenever $Q$ is a non-trivial convex compact $G$-space in which $L$ fixes a point, $Q$ is not irreducible. This is indeed a weakening of the notion of co-amenability \footnote[2]{and not a relative version of a notion of weak amenability.}, which asks that every convex compact $G$-space $Q$ with $L$-fixed points has $G$-fixed points \cite{Eymard-moy} (and hence $Q$ is not irreducible, unless trivial). If $G$ has a subgroup that is both amenable and weakly co-amenable, then $G$ is amenable; and a \textit{normal} weakly co-amenable subgroup is co-amenable. However in general weak co-amenability does not imply co-amenability, even for discrete groups. In \S \ref{subsec-bnd-g(f,f')} we exhibit examples of finitely generated groups such that every subgroup is either amenable or weakly co-amenable, but having non-amenable subgroups that are not co-amenable.

Finally we say that a subgroup $L \leq G$ is \textbf{boundary-minimal} if there exists a non-trivial $G$-boundary on which $L$ acts minimally. We refer to \S \ref{subsec-normal-bound} for context and examples.

\begin{thm} \label{thmintro-normal-coamenable-lc}
Let $H$ be a locally compact group with an amenable URS that comes from an extremely proximal action, and whose envelope is co-amenable in $H$. Let $G$ be a locally compact group containing $H$ as a closed subgroup of finite covolume. Then $G$ is boundary indivisible, and the following hold:
\begin{enumerate}[label=(\alph*)]
\item \label{item-intr-N-am-coam} Every closed normal subgroup of $G$ is either amenable or co-amenable.
\item \label{item-intr-urs-wco} If $L$ is a boundary-minimal subgroup of $G$, and $L$ is uniformly recurrent, then $L$ is weakly co-amenable in $G$. 
\end{enumerate}
\end{thm}

Again we make several comments:

\begin{enumerate}[itemsep=1.7ex, leftmargin = 0.7cm, label=\arabic*)]

\item The group $H$ is allowed to be discrete, so the theorem applies for all groups $\Gamma$ as in Theorem \ref{thm-intro-URS-lattice-prod}. While boundary indivisibility of $G$ will be an intermediate step in the proof of Theorem \ref{thm-intro-URS-lattice-prod}, statements \ref{item-intr-N-am-coam} and \ref{item-intr-urs-wco} provide additional information that is rather independent of the conclusion of Theorem \ref{thm-intro-URS-lattice-prod}.

\item  \ref{item-intr-N-am-coam} will be deduced from \ref{item-intr-urs-wco}, as we will prove that any non-amenable normal subgroup is boundary-minimal (see Theorem \ref{thm-normal-nonam-Gbnd}).

\item If $\mathrm{Rad}(G)$ denotes the amenable radical of $G$, statement \ref{item-intr-N-am-coam} could be equivalently stated saying that $G/\mathrm{Rad}(G)$ is a just non-amenable group, i.e.\ $G/\mathrm{Rad}(G)$ is non-amenable and every proper quotient is amenable.

\item \label{item-cooment-intro-rad} Theorem \ref{thmintro-normal-coamenable-lc} does not say anything about amenable normal subgroups of $G$. It is worst pointing out that, as illustrated by the examples discussed in \cite{LB-irr-wreath}, it happens that a discrete group $\Gamma$ satisfying the assumptions of Theorem \ref{thmintro-normal-coamenable-lc} and with trivial amenable radical, sits as a lattice in a group $G$ with non-compact (e.g.\ infinite discrete) amenable radical.

\item \label{item-cooment-intro-coam} Remark \ref{rmq-urs-wco-not-co} below provides counter-examples showing that in statement \ref{item-intr-urs-wco} the conclusion cannot be strengthened by saying that $L$ is co-amenable in $G$.

\end{enumerate}

We view the above remarks \ref{item-cooment-intro-rad}-\ref{item-cooment-intro-coam} as illustrations of the limitations of the use of topological boundaries and URS's to the problem addressed here in the rather abstract setting of Theorem \ref{thmintro-normal-coamenable-lc}.

\medskip

Group actions on trees are a natural source of extremely proximal actions, and Theorems \ref{thm-intro-URS-lattice-prod} and \ref{thmintro-normal-coamenable-lc} find applications in this setting. In the following statement $T$ is a locally finite simplicial tree.

\begin{cor} \label{cor-intro-tree-disc}
Let $\Gamma \leq \mathrm{Aut}(T)$ be a countable group having no proper invariant subtree and no finite orbit in $T \cup \partial T$. Assume that $\Gamma_\xi$ is non-trivial and amenable for all $\xi \in \partial T$; and $\Gamma$ is virtually simple. If $G$ is a locally compact group containing $\Gamma$ as a lattice, then:
\begin{enumerate}[label=(\alph*)]
\item any continuous morphism with dense image $G \rightarrow G_1 \times G_2$ is such that one factor $G_i$ is compact. In particular $G$ itself cannot be a direct product of two non-compact groups.
\item Every closed normal subgroup of $G$ is either amenable or co-amenable.
\item If $L$ is a boundary-minimal subgroup of $G$, and $L$ is uniformly recurrent, then $L$ is weakly co-amenable in $G$. 
\end{enumerate}
\end{cor}

A group $\Gamma$ as in Corollary \ref{cor-intro-tree-disc} is never discrete in $\mathrm{Aut}(T)$. Recall that Burger and Mozes constructed simple groups $\Gamma$ acting on two locally finite regular trees $T,T'$ such that the image of $\Gamma$ in $\mathrm{Aut}(T)$ and $\mathrm{Aut}(T')$ are non-discrete, but $\Gamma$ acts freely and cocompactly on $T \times T'$, so that $\Gamma$ is a cocompact lattice in the product $\mathrm{Aut}(T) \times \mathrm{Aut}(T')$ \cite{BM-IHES-2}. These examples illustrate the fact that the assumption in Corollary \ref{cor-intro-tree-disc} that end-stabilizers are \textit{all} non-trivial is essential.

Examples of groups to which Corollary \ref{cor-intro-tree-disc} applies can be found among the family of groups denoted $G(F,F')$ in \cite{LB-ae} (see Corollary \ref{cor-g(f,f')-lat-obstruc}). These are examples of groups with a continuous Furstenberg URS. Here $F' \leq \mathrm{Sym}(d)$ is a finite permutation group and $F$ is a regular subgroup of $F'$. Recall that a permutation group is \textbf{regular} if it acts freely and transitively. The group $G(F,F')$ is then a finitely generated group acting on a $d$-regular tree, transitively on vertices and edges, and with local action at every vertex isomorphic to $F'$. We refer to \S \ref{subsec-background-g(f,f')} for a definition. The normal subgroup structure of these groups is highly sensible to the permutation groups: there are permutation groups $F,F'$ such that $G(F,F')$ virtually admits a non-abelian free quotient (Proposition \ref{prop-explicit-G/G^+}), and there are permutation groups $F,F'$ such that $G(F,F')^\ast$ (the subgroup of index two in $G(F,F')$ preserving the bipartition of $T_d$) is simple \cite[Cor.\ 4.14]{LB-ae}. This family of groups and the family of Burger--Mozes lattices in the product of two trees both contain instances of finitely generated simple groups which embed densely in some universal group $U(F)^+$ \cite[\S 3.2]{BM-IHES}. Despite these similarities, Corollary \ref{cor-g(f,f')-lat-obstruc} shows that any group containing a virtually simple $G(F,F')$ as a lattice is rather allergic to any direct product behavior.

\begin{rmq} \label{rmq-short}
As mentioned earlier, examples of lattice embeddings for the groups $G(F,F')$ are described in \cite{LB-irr-wreath}. There it is shown that these groups do embed as irreducible lattices in locally compact groups, but these locally compact groups are wreath products rather than direct products. 
\end{rmq}

We also mention that other examples of groups to which Corollary \ref{cor-intro-tree-disc} can be applied may be found among the family of piecewise prescribed tree automorphism groups considered in \cite[Sec.\ 4]{LB-cs}.

\subsection*{Questions}

We end this introduction with two questions. Extreme proximality is used in a crucial way at different stages of the proofs of Theorems \ref{thm-intro-URS-lattice-prod} and \ref{thmintro-normal-coamenable-lc}. These results both fail without the extreme proximality assumption, simply because then the group itself may very well be a direct product. Putting aside these trivial counter-examples, we do not know whether serious algebraic restrictions on a locally compact group may be derived from the existence of a lattice with a continuous Furstenberg URS. In this direction, we find the following question natural:

\begin{quest} \label{quest-latt-prd}
Does there exist $\Gamma$ with a continuous Furstenberg URS which is a lattice in a group $G = G_1 \times G_2$ such both factors are non-discrete, and $\Gamma$ has an injective and dense projection to each factor ? What if we impose moreover that $\Gamma$ has trivial amenable radical ?
\end{quest}

 Theorem 2.8 from \cite{LB-irr-wreath} presents a situation of a locally compact group $G$ with two cocompact lattices $\Gamma_1, \Gamma_2 \leq G$ such that the stabilizer URS associated to the $\Gamma_1$-action on $\partial_{sp} G$ is $\left\{\mathrm{Rad}(\Gamma_1)\right\}$, while the stabilizer URS associated to the $\Gamma_2$-action on $\partial_{sp} G$ is continuous. Here $\partial_{sp} G$ stands for the Furstenberg boundary of $G$; see \S \ref{subsec-top-bnd}. In these examples the group $G$ splits as $G = N \rtimes Q$, where $N$ is the amenable radical of $G$. The lattice $\Gamma_1$ preserves this splitting, meaning that we have $\Gamma_1 = (N \cap \Gamma_1) \rtimes (Q \cap \Gamma_1)$ (and hence $\Gamma_1$ does not act faithfully on $\partial_{sp} G$), while $\Gamma_2$ has an injective projection to $Q$. This naturally raises the following:

\begin{quest}
Let $G$ be a locally compact group with two lattices $\Gamma_1$ and $\Gamma_2$ both acting faithfully on $X = \partial_{sp} G$. Is it possible that the $\Gamma_1$-action on $X$ is topologically free, but the $\Gamma_2$-action on $X$ is not topologically free ? Can this happen with $\partial_{sp} G = G/H$ being a homogeneous $G$-space ?
\end{quest}

Note that by \cite[Prop.\ 7]{Furman-min-strg}, the condition that $\Gamma_1$ and $\Gamma_2$ act faithfully on $\partial_{sp} G$ is equivalent to saying that $\Gamma_1$ and $\Gamma_2$ have trivial amenable radical. Recall that topologically free means that there is a dense subset of points having trivial stabilizer (equivalently, the stabilizer URS is trivial).

\subsection*{Outline of proofs and organization} 

The article is organized as follows. In the next section we introduce terminology and preliminary results about topological boundaries and extremely proximal actions. In Section \ref{sec-urs} we establish the results about uniformly recurrent subgroups that are used in later sections. In particular we prove a certain gap property for URS's coming from extremely proximal actions (Proposition \ref{prop-dicho-urs-EP}). Combined with an observation about compact spaces with comparable stabilizer URS's (Proposition \ref{prop-same-urs}), we deduce that a locally compact group $H$ with an amenable URS that comes from an extremely proximal action, and whose envelope is co-amenable in $H$, is boundary indivisible (Proposition \ref{prop-A_G-EP}).

The setting of Section \ref{sec-epntf} is that of a group admitting a non-topologically free extremely proximal action. We establish intermediate results, notably concerning normal subgroups (Proposition \ref{prop-core-deriv0}) and commensurated subgroups (Proposition \ref{prop-commens-dicho}), and deduce non-embedding results for this class of groups (see Proposition \ref{prop-lin-rep-core} and Corollary \ref{cor-epntf-embed-amen}).

In Section \ref{sec-lattices-urs} we use results from Section \ref{sec-urs} together with Proposition \ref{prop-furst} of Furstenberg and prove Theorem \ref{thmintro-normal-coamenable-lc}. We then specify to discrete groups and give the proof of Theorem \ref{thm-intro-URS-lattice-prod}. The proof essentially splits in two steps: the first one is the application of Theorem \ref{thmintro-normal-coamenable-lc} to obtain amenability of one factor, and the second consists in proving that under appropriate assumptions the amenable factor is compact, using results from Section \ref{sec-epntf}.

In Section \ref{sec-trees} we consider groups acting on trees, and apply previous results of the article to this setting. After giving the proof of Corollary \ref{cor-intro-tree-disc}, we focus on the family of groups with prescribed local action $G(F,F')$. We study boundaries of these groups, and use results from Section \ref{sec-urs} in order to characterize the discrete groups within this family which are boundary indivisible (see Theorem \ref{thm-boundaries-G(F,F')}). This includes those which are virtually simple, but this also contain non-virtually simple instances.

\subsection*{Acknowledgements}

I am grateful to Alex Furman for pointing out Proposition \ref{prop-furst} to my attention, and to Uri Bader for enlightening discussion about the proof. I am also grateful to Pierre-Emmanuel Caprace, Yves Cornulier, Bruno Duchesne, Nicol\'{a}s Matte Bon, Nicolas Monod and Pierre Pansu for interesting discussions and comments related to this work; and to a referee for a careful reading of the article and for useful comments.

\section{Preliminaries}

\subsection{Conventions and terminology}

The letter $G$ will usually refer to a topological group, while $\Gamma$ will denote a discrete group. The group of homeomorphic automorphisms of $G$ will be denoted $\mathrm{Aut}(G)$. Whenever $G$ is a locally compact group, we will always assume that $G$ is second countable.

The notation $\X$ will refer to a topological space. The letters $X,Y$ will be reserved for compact spaces, and $Z$ for a compact space equipped with an extremely proximal group action. All compact spaces are assumed to be Hausdorff. 

A space $\X$ is a $G$-space if $G$ admits a continuous action $G \times \X \rightarrow \X$.  The action of $G$ on $\X$ (or the $G$-space $\X$) is \textbf{minimal} if all orbits are dense. The $G$-space $\X$ is said to be trivial if $\X$ is a one-point space.

If $\X$ is locally compact, we denote by $\prob(\X)$ the set of all regular Borel probability measures on $\X$. The space of continuous compactly supported functions on $\X$ is denoted $C_K(\X)$. Each $\mu \in \prob(\X)$ defines a linear functional on $C_K(\X)$, and we endow $\prob(\X)$ with the weak*-topology: a net $(\mu_i)$ converges to $\mu$ if $\mu_i(f) \rightarrow \mu(f)$ for all $f \in C_K(\X)$. By Banach-Alaoglu theorem, $\prob(\X)$ is relatively compact in $C_K(\X)^*$.

We denote by $2^\X$ the set of all closed subsets of $\X$. The sets \[ O(K;U_1,\ldots,U_n) = \left\{C \in 2^\X \, : \, C \cap K = \emptyset; \, C \cap U_i \neq \emptyset \, \, \text{for all $i$} \, \right\}, \] where $K \subset \X$ is compact and $U_1,\ldots,U_n \subset \X$ are open, form a basis for the Chabauty topology on $2^\X$. Endowed with the Chabauty topology, the space $2^\X$ is compact. We will freely identify $\X$ with its image in $2^\X$ by the natural inclusion $x \mapsto \left\{x\right\}$. Note that when $\X$ is a $G$-space, so is $2^\X$.

In the particular case where $\X = G$ is a locally compact group, the space $\sub(G)$ of closed subgroups of $G$ is closed in $2^G$. In particular $\sub(G)$ is a compact space, on which $G$ acts by conjugation. A \textbf{uniformly recurrent subgroup} (URS) of $G$ is a closed, $G$-invariant, minimal subset of $\sub(G)$. The set of URS's of $G$ is denoted $\urs(G)$. By extension we also say that a subgroup $H \leq G$ is uniformly recurrent if the closure of the conjugacy class of $H$ in $\sub(G)$ is minimal.

\subsection{Topological boundaries} \label{subsec-top-bnd}

Let $X$ be a compact $G$-space. The action of $G$ on $X$ is \textbf{proximal} if the closure of any $G$-orbit in $X \times X$ intersects the diagonal. The $G$-action on $X$ is \textbf{strongly proximal} if the closure of any $G$-orbit in $\prob(X)$ contains a Dirac measure. Strong proximality is stable under taking products (with diagonal action) and continuous equivariant images (see e.g.\ \cite{Gl-PF}).

We say that $X$ is a \textbf{boundary} if $X$ is both minimal and strongly proximal. For every topological group $G$, there exists a unique boundary $\partial_{sp} G$ with the universal property that for any boundary $X$, there exists a continuous $G$-equivariant surjection $\partial_{sp} G \rightarrow X$ \cite[Prop.\ 4.6]{Furst-bd-th}. This universal space $\partial_{sp} G$ is referred to as the \textbf{Furstenberg boundary} of $G$. It is easy to verify that any amenable normal subgroup $N$ of $G$ acts trivially on any $G$-boundary, so that $\partial_{sp} G = \partial_{sp} (G/N)$. 

If $G$ admits a cocompact amenable subgroup, then the Furstenberg boundary is a homogeneous space $\partial_{sp} G = G/H$, and the $G$-spaces of the form $G/L$ with $L$ containing $H$ are precisely the $G$-boundaries \cite[Prop.\ 4.4]{Furst-bd-th}. The situation for discrete groups is quite different: as shown in \cite{KK} and \cite{BKKO}, Furstenberg boundaries of discrete groups are always non-metrizable (unless trivial).

The following is a fundamental property of boundaries (see \cite[III.2.3]{Gl-PF}):

\begin{thm} \label{thm-conv-bnd}
Any convex compact $G$-space contains a boundary. In fact if $Q$ is an irreducible convex compact $G$-space, then the action of $G$ on $Q$ is strongly proximal, and the closure of extreme points of $Q$ is a $G$-boundary.
\end{thm}

Irreducible means that $Q$ has no proper closed convex $G$-invariant subspace. In particular Theorem \ref{thm-conv-bnd} has the following consequence (\cite[III.3.1]{Gl-PF}):

\begin{thm} \label{thm-amean-bnd}
A group $G$ is amenable if and only if all $G$-boundaries are trivial, or equivalently $\partial_{sp} G$ is trivial.
\end{thm}

\subsection{Extremely proximal actions} \label{subsec-EP}

Let $X$ be a compact $G$-space. A closed subset $C$ of $X$ is \textbf{compressible} if the closure of the $G$-orbit of $C$ in the space $2^X$ contains a singleton $\left\{x\right\}$. Equivalently, for every neighbourhood $U$ of $x$, there exists $g \in G$ such that $g(C) \subset U$. The action of $G$ on $X$ is \textbf{extremely proximal} if every closed subset $C \subsetneq X$ is compressible. References where extremely proximal actions were considered include \cite{Glas-top-dyn-group,Laca-Spiel,Jo-Rob-simple-pur,shift-EP,LBMB}.

We will make use of the following result, which is Theorem 2.3 from \cite{Glas-top-dyn-group}:

\begin{thm} \label{thm-ep-sp}
Let $X$ be a compact $G$-space, and assume $X$ has at least three points. If the $G$-action on $X$ is extremely proximal, then it is strongly proximal. 
\end{thm}

Examples of extremely proximal actions are provided by group actions on trees or hyperbolic spaces. If $G \leq \mathrm{Aut}(T)$ acts on $T$ with no proper invariant subtree and no finite orbit in $T \cup \partial T$, then the action of $G$ on $\partial T$ is minimal and extremely proximal; and if $G$ acts coboundedly on a proper geodesic hyperbolic space $\mathbb{X}$ with no fixed point or fixed pair at infinity, then the $G$-action on the Gromov boundary $\partial \mathbb{X}$ is minimal and extremely proximal.

These two situations are particular cases of the following more general result, that we believe is well-known. A homeomorphism $g$ of a space $X$ is hyperbolic if there exist $\xi_-, \xi_+ \in X$, called the endpoints of $g$, such that for all neighbourhoods $U_-,U_+$ of $\xi_-, \xi_+$, for $n$ large enough we have $g^n(X \setminus U_-) \subset U_+$ and $g^{-n}(X \setminus U_+) \subset U_-$.

\begin{prop} \label{prop-folk-EP}
If $G$ acts on a compact space $X$ with hyperbolic elements having no common endpoints, and such that the set of endpoints of hyperbolic elements of $G$ is dense in $X$, then the action is minimal and extremely proximal.
\end{prop}

\begin{proof}
Let $U \subset X$ be a non-empty open invariant subset. By our density assumption, there is $g \in G$ hyperbolic whose attracting endpoint $\xi_+$ belongs to $U$. So for every $x \neq \xi_-$, there is $n > 0$ such that $g^n(x) \in U$ since $U$ is open, so we deduce that $U$ contains $X \setminus \left\{ \xi_- \right\}$. But the existence of hyperbolic elements with no common endpoints ensures that $G$ fixes no point of $X$, so finally $U = X$, i.e.\ the action is minimal.

Now if $C \subsetneq X$ is a closed subset then again there is $g \in G$ whose attracting endpoint is outside $C$, and $C$ is compressible to the repealing endpoint of $g$.
\end{proof}

Recent work of Duchesne and Monod shows that group actions on dendrites is also a source of extremely proximal actions. Recall that a dendrite $X$ is a compact metrizable space such that any two points are the extremities of a unique arc. Duchesne and Monod show that if $\Gamma$ acts on $X$ with no invariant proper sub-dendrite, then there is a unique minimal closed invariant subset $M \subseteq X$ and the $\Gamma$-action on $M$ is extremely proximal. See the proof of Theorem 10.1 in \cite{DM-dend-curv}.

Extremely proximal actions also play a prominent role in the context of group actions on the circle. For any minimal action $\alpha: \Gamma \rightarrow \mathrm{Homeo}^+(\mathbb{S}^1)$, either $\alpha(\Gamma)$ is conjugated to a group of rotations, or $\alpha(\Gamma)$ has a finite centralizer $C_\Gamma$ in $\mathrm{Homeo}^+(\mathbb{S}^1)$ and the action of $\Gamma$ on the quotient circle $C_\Gamma \backslash \mathbb{S}^1$ is extremely proximal: see Ghys \cite{Ghys-circle} and Margulis \cite{Marg-free-circle}. We mention however that in all the examples of countable groups $\Gamma$ with an action on $\mathbb{S}^1$ that is minimal and not topologically free that we are aware of, the stabilizer URS is either non-amenable, or not known to be amenable. In particular we do not know any application of Theorem \ref{thm-intro-URS-lattice-prod} to groups acting on the circle.

In the sequel we will make use of the following easy lemma.

\begin{lem} \label{lem-EP-to-cocompact}
Let $G$ be a topological group, and $H$ a subgroup of $G$ such that there is some compact subset $K$ of $G$ such that $G = K H$. Let $X$ be a compact $G$-space, and $C$ a closed subset of $X$ that is compressible by $G$. Then $C$ is compressible by $H$. In particular if the $G$-action on $X$ is extremely proximal, then the $H$-action is extremely proximal.
\end{lem}

\begin{proof}
By assumption there exists $x \in X$ and $(g_i)$ such that $g_i(C)$ converges to $x$ in $2^X$. If $g_i = k_i h_i$, by compactness of $K$ we assume that $(k_i)$ converges to some $k$, and it follows that $h_i(C)$ converges to $k^{-1}x$ by continuity of $G \times 2^X \rightarrow 2^X$.
\end{proof}

\section{Uniformly recurrent subgroups} \label{sec-urs}

\subsection{Generalities on uniformly recurrent subgroups} \label{subsec-gen-urs}

Let $G$ be a locally compact group. For $\H, \K \in \urs(G)$, we write $\H \preccurlyeq \K$ when there exist $H \in \H$ and $K \in \K$ such that $H \leq K$. This is equivalent to the fact that every $H \in \H$ is contained in an element of $\K$, and every $K \in \K$ contains an element of $\H$, and the relation $\preccurlyeq$ is an order on $\urs(G)$. See e.g.\ \S 2.4 in \cite{LBMB}.

For simplicity the URS $\left\{N\right\}$ associated to a closed normal subgroup $N$ of $G$ will still be denoted $N$. In particular $N \preccurlyeq \H$ (resp.\ $\H \preccurlyeq N$) means that $N$ is contained in (resp.\ contains) all the elements of $\H$. By the trivial URS we mean the URS corresponding to the trivial subgroup $\left\{1\right\}$. We warn the reader that in this terminology the URS corresponding to a non-trivial normal subgroup $N$ is trivial as a $G$-space (it is a one-point space), but is not trivial as a URS.

Let $X,Y$ be compact $G$-spaces. We say that $X$ is a \textbf{factor} of $Y$, and $Y$ is an \textbf{extension} of $X$, if there exists a continuous equivariant map $Y \rightarrow X$ that is onto. If $\pi: Y \rightarrow X$ is a continuous equivariant map, we say that $\pi$ is \textbf{almost 1-1} if the set of $y \in Y$ such that $\pi^{-1}(\pi(y)) = \left\{y\right\}$ is dense in $Y$. When moreover $\pi$ is onto we say that $Y$ is an almost 1-1 extension of $X$. 


We now recall the definition of the stabilizer URS associated to a minimal action on a compact space. If $X$ is a compact $G$-space and $x \in X$, we denote by $G_x$ the stabilizer of $x$ in $G$.

\begin{defi}
If $X$ is a compact $G$-space, we denote by $X_0 \subset X$ the set of points at which $\mathrm{Stab}: X \rightarrow \sub(G)$, $x \mapsto G_x$, is continuous.
\end{defi}

Upper semi-continuity of the map $\mathrm{Stab}$ and second countability of $G$ imply that $X_0$ is a dense subset of $X$ (indeed if $(U_n)$ is a basis of the topology on $G$ and $X_n$ is the set of $x \in X$ such that $G_x \cap \overline{U_n} \neq \emptyset$, which is closed, one verifies that $\mathrm{Stab}$ is continuous on $\cap_n (\partial X_n)^c$). Following \cite{Gla-Wei}, we denote \[ \tilde{X} = \mathrm{cls}\left\{(x,G_x) \, : \, x \in X_0\right\} \subset X \times \sub(G),\] and \[ \mathcal{S}_G(X) = \mathrm{cls}\left\{G_x \, : \, x \in X_0\right\} \subset \sub(G),\] where $\mathrm{cls}$ stands for the closure in the ambient space. We have the obvious inclusions \[ \tilde{X} \subseteq X^\ast := \mathrm{cls}\left\{(x,G_x) \, : \, x \in X\right\} \] and \[ \mathcal{S}_G(X) \subseteq \mathcal{S}_G^\ast(X) := \mathrm{cls}\left\{G_x \, : \, x \in X\right\}. \]

We denote by $\eta_X$ and $\pi_X$ the projections from $X \times \sub(G)$ to $X$ and $\sub(G)$ respectively.

\begin{prop}[Prop.\ 1.2 in \cite{Gla-Wei}] \label{prop-GW-1.2}
If $X$ is a minimal compact $G$-space, then $\eta_X: \tilde{X} \rightarrow X$ is an almost 1-1 extension, and $\tilde{X}$ and $\mathcal{S}_G(X)$ are the unique minimal closed $G$-invariant subsets of respectively $X^\ast$ and $\mathcal{S}_G^\ast(X)$.
\end{prop}

\begin{defi} \label{def-stab-urs}
If $X$ is a minimal compact $G$-space, $\mathcal{S}_G(X)$ is the \textbf{stabilizer URS} associated to the $G$-action on $X$. The action of $G$ on $X$ is \textbf{topologically free} if $\mathcal{S}_G(X)$ is trivial, i.e.\ $\mathcal{S}_G(X) = \left\{\left\{1\right\}\right\}$.
\end{defi}

\begin{rmq}
When $G$ is not assumed second countable, in general $X_0$ is no longer dense in $X$. However it is still possible to define the stabilizer URS associated to a minimal action on a compact space; see the discussion in \cite[p1]{MB-Ts-realizing}.
\end{rmq}

In the sequel we will sometimes use the following version of Proposition \ref{prop-GW-1.2}.

\begin{prop} \label{prop-mini-urs-min}
Let $X$ be compact $G$-space, and let $H \leq G$ be a subgroup acting minimally on $X$. Then $H$ acts minimally on $\tilde{X}$ and $\mathcal{S}_G(X)$, and $\tilde{X}$ and $\mathcal{S}_G(X)$ are the unique minimal closed $H$-invariant subsets of $X^\ast$ and $\mathcal{S}_G^\ast(X)$.
\end{prop}

\begin{proof}
Let $Y \subseteq X^\ast$ closed and $H$-invariant. Since $X$ is a factor of $X^\ast$ and $H$ acts minimally on $X$, for every $x \in X_0$ there exists $L \in \sub(G)$ such that $(x,L) \in Y$. But for $x \in X_0$, the fact that $(x,L)$ belongs to $X^\ast$ forces $L$ to be equal to $G_x$ by definition of $X_0$, and it follows that $\tilde{X} \subseteq Y$. Moreover $H$ acts minimally on $\tilde{X}$ since $\eta_X: \tilde{X} \rightarrow X$ is an almost 1-1 extension and minimality is preserved by taking almost 1-1 extensions (if $\pi: X_1 \rightarrow X_2$ is almost 1-1 and if $C \subseteq X_1$ is a closed subset such that $\pi(C) = X_2$, then $C = X_1$). So the statements for $\tilde{X}$ and $X^\ast$ are established, and the same hold for $\mathcal{S}_G(X)$ and $\mathcal{S}_G^\ast(X)$ since these are factors of $\tilde{X}$ and $X^\ast$.
\end{proof}

\subsection{Envelopes}

Let $G$ be a locally compact group and $\H \in \urs(G)$.

\begin{defi}
The \textbf{envelope} $\env(\H)$ of $\H$ is the closed subgroup of $G$ generated by all the subgroups $H \in \H$.
\end{defi}

By definition $\env(\H)$ is the smallest closed subgroup of $G$ such that $\H \subset \sub(\env(\H))$. Note that $\env(\H)$ is a normal subgroup of $G$, and is actually the smallest normal subgroup such that $\H \preccurlyeq \env(\H)$. 

\medskip

Let $\Gamma$ be a discrete group, $X$ a compact $\Gamma$-space and $X_0$ the domain of continuity of the map $\mathrm{Stab}$. It is a classical fact that $X_0$ consists of those $x \in X$ such that for every $\gamma \in \Gamma_x$, there exists $U$ neighbourhood of $x$ that is fixed by $\gamma$ (see e.g.\ \cite[Lem.\ 5.4]{Vorob-grig} for a proof). For $x \in X$, we will denote by $\Gamma_x^0 \leq \Gamma_x$ the set of elements fixing a neighbourhood of $x$, so that $x \in X_0$ if and only if $\Gamma_x = \Gamma_x^0$.

\begin{lem} \label{lem-equiv-fix-pt}
Let $\Gamma$ be a countable discrete group, $X$ a compact minimal $\Gamma$-space, $n \geq 1$ and $\gamma_1,\ldots,\gamma_n \in \Gamma$. The following are equivalent:
\begin{enumerate}[label=(\roman*)]
	\item \label{item-int} $\cap \fix(\gamma_i)$ has non-empty interior;
	\item \label{item-X} there is $x \in X$ such that $\gamma_i \in \Gamma_x^0$ for all $i$;
	\item \label{item-X0} same as in \ref{item-X} but with $x \in X_0$;
	\item \label{item-urs} there is $H \in S_\Gamma(X)$ such that $\gamma_i \in H$ for all $i$.
\end{enumerate}
In particular $\env(\mathcal{S}_\Gamma(X))$ is generated by the elements $\gamma \in \Gamma$ such that $\fix(\gamma)$ has non-empty interior.
\end{lem}

\begin{proof}
It is clear that \ref{item-int} and \ref{item-X} are equivalent. Also \ref{item-X0} clearly implies \ref{item-X}, and \ref{item-X} also implies \ref{item-X0} by density of $X_0$ in $X$. Finally \ref{item-X0} implies \ref{item-urs} since $\Gamma_x^0 \in \mathcal{S}_\Gamma(X)$ for $x \in X_0$, and \ref{item-urs} implies \ref{item-X0} by density of the set of $\Gamma_x^0$, $x \in X_0$, in $\mathcal{S}_\Gamma(X)$.
\end{proof}

\subsection{$G$-spaces with comparable stabilizer URS's} \label{subsec-presc-urs}

We recall the notion of disjointness from \cite{Furst-disjoint}. Two compact $G$-spaces $X,Y$ are \textbf{disjoint} if whenever $X,Y$ are factors of a compact $G$-space $\Omega$ via $\varphi_X: \Omega \rightarrow X$ and $\varphi_Y: \Omega \rightarrow Y$, then the map $(\varphi_X, \varphi_Y): \Omega \rightarrow X \times Y$ is surjective. When $X,Y$ are minimal $G$-spaces, this is equivalent to saying that the product $X \times Y$ remains minimal \cite[Lem.\ II.1]{Furst-disjoint}.

The following lemma presents a situation which easily implies disjointness:

\begin{lem} \label{lem-triv-disj}
Let $X,Y$ be minimal compact $G$-spaces such that there exists $y_0 \in Y$ such that $G_{y_0}$ acts minimally on $X$. Then $X$ and $Y$ are disjoint.
\end{lem}

\begin{proof}
This is clear: if $W$ is a closed invariant subset of $X \times Y$, then by minimality of $Y$ there exists $x_0 \in X$ such that $(x_0,y_0) \in W$. Since $G_{y_0}$ acts minimally on $X$ we deduce that $W$ contains $X \times \left\{y_0\right\}$, and by minimality of $Y$ it follows that $W$ is equal to $X \times Y$.
\end{proof}

The following proposition will be used notably in Proposition \ref{prop-A_G-EP}.

\begin{prop} \label{prop-same-urs}
Let $X,Y$ be compact minimal $G$-spaces, and write $\H = \mathcal{S}_G(X)$ and $\K = \mathcal{S}_G(Y)$. Suppose that $\H \preccurlyeq \K$. Then $X$ and $Y$ can be disjoint only if $\env(\H) \preccurlyeq \K$.

In particular if $\mathcal{S}_G(X) = \mathcal{S}_G(Y)$ and this URS is not a point, then $X$ and $Y$ are not disjoint.
\end{prop}

\begin{proof}
Using notation from Proposition \ref{prop-GW-1.2}, we have almost 1-1 extensions $\eta_X : \tilde{X} \rightarrow X$ and $\eta_Y : \tilde{Y} \rightarrow Y$, and we write $\eta = \eta_X \times \eta_Y: \tilde{X} \times \tilde{Y} \rightarrow X \times Y$, and $\pi = \pi_X \times \pi_Y: \tilde{X} \times \tilde{Y} \rightarrow \H \times \K$. The set $W \subseteq \H \times \K$ of pairs $(H,K)$ such that $H \leq K$ is non-empty by assumption, is clearly $G$-invariant, and is easily seen to be closed. If $W$ is a proper subset of $\H \times \K$ then $\pi^{-1}(W)$ is a proper subset of $\tilde{X} \times \tilde{Y}$ since $\pi$ is a factor, and it follows that $\eta \left( \pi^{-1}(W) \right)$ is a closed $G$-invariant subset of $X \times Y$ that is proper since $\eta$ is almost 1-1. This contradicts disjointness of $X$ and $Y$. Therefore we have $W = \H \times \K$. This means that for a fixed $K \in \K$ we have $H \leq K$ for every $H \in \H$, and hence $\env(\H) \preccurlyeq \K$.
\end{proof}

\subsection{Action of a URS on a $G$-space}

In this paragraph $G$ still denote a locally compact group, and $X$ is a compact $G$-space. Given $\H \in \urs(G)$, we study the properties of the action of elements on $\H$ on the space $X$.

The proof of the following lemma is an easy verification, and we leave it to the reader.

\begin{lem}
If $X$ is a compact $G$-space, $\left\{H \in \sub(G) \, | \, H \, \, \text{fixes a point in} \, \, X\right\}$ is a closed $G$-invariant subset of $\sub(G)$.
\end{lem}

In particular the following definition makes sense.

\begin{defi}
Let $X$ be a compact $G$-space, and $\H \in \urs(G)$. We say that $\H$ fixes a point in $X$ if for some (all) $H \in \H$, there is $x \in X$ such that $h(x)=x$ for all $h \in H$.
\end{defi}

\begin{lem} \label{lem-fix-pt-min}
Let $X$ be a compact $G$-space, $Y \subseteq X$ a closed invariant subset of $X$, and $\H \in \urs(G)$. If there exists $H \in \H$ fixing $x \in X$ such that $\overline{G x} \cap Y \neq \varnothing$, then $\H$ fixes a point in $Y$. 
\end{lem}

\begin{proof}
By assumption there exist $(g_i)$ and $y \in Y$ such that $g_i(x)$ converges to $y$. If $K \in \H$ is a limit point of $(H^{g_i})$ (which exists by compactness), then $K$ fixes $y$ by upper semi-continuity of the stabilizer map. 
\end{proof}

Lemma \ref{lem-fix-pt-min} implies the following:

\begin{lem} \label{lem-fp-Xmin}
If $X$ is a compact $G$-space containing a unique minimal closed $G$-invariant subset $X_{min} \subset X$ (e.g.\ $X$ is proximal), and if $\H \in \urs(G)$ fixes a point in $X$, then $\H$ fixes a point in $X_{min}$.
\end{lem}

\begin{prop} \label{prop-fix-pt-min}
Let $Z$ be a compact $G$-space that is extremely proximal, and $\H \in \urs(G)$. Then either $\H$ fixes a point in $Z$, or all $H \in \H$ act minimally on $Z$. 
\end{prop}

\begin{proof}
If there exist $H \in \H$ and a non-empty closed subset $C \subsetneq Z$ that is invariant by $H$, then we may apply Lemma \ref{lem-fix-pt-min} to the space $X = 2^Z$, the subspace $Y = Z$ and the point $x = C$, and we deduce that $\H$ fixes a point in $Z$.
\end{proof}

Recall that given a compact $G$-space $X$, $\mathcal{S}_G^\ast(X)$ stands for the closure in $\sub(G)$ of the set of subgroups $G_x$, $x \in X$.

\begin{lem} \label{lem-urs-acts-min}
Let $X$ be a compact $G$-space. Assume that $K \leq G$ is a closed subgroup of $G$ which acts minimally on $X$ and such that there exists $H \in \mathcal{S}_G^\ast(X)$ with $H \leq K$. Then $\env(\mathcal{S}_G(X)) \leq K$.
\end{lem}

\begin{proof}
Since $K$ acts minimally on $X$, the closure of the $K$-orbit of $H$ in $\mathcal{S}_G^\ast(X)$ contains $\mathcal{S}_G(X)$ according to Proposition \ref{prop-mini-urs-min}. Since $H \in \sub(K)$ and $\sub(K)$ is a closed subset of $\sub(G)$, we deduce that $\mathcal{S}_G(X) \subset \sub(K)$, and in particular $\env(\mathcal{S}_G(X)) \leq K$.
\end{proof}

\begin{defi}
Let $\H \in \urs(G)$. We say that $\H$ \textbf{comes from an extremely proximal action} if there exists a compact $G$-space $Z$ that is minimal and extremely proximal, and such that $\mathcal{S}_G(Z) = \H$.
\end{defi}

It was shown in \cite{LBMB} that for a discrete group $\Gamma$ with a non-trivial URS $\H$ coming from an extremely proximal action, any non-trivial $\K \in \urs(\Gamma)$ must be \enquote{relatively large} with respect to $\H$ (see \cite[Th.\ 3.10]{LBMB} for a precise statement). Appropriate assumptions on $\Gamma$ and $\H$ further imply that $\H \preccurlyeq \K$ for every non-trivial $\K \in \urs(\Gamma)$ \cite[Cor.\ 3.12]{LBMB}. The following proposition goes in the opposite direction by considering URS's larger than $\H$.

\begin{prop} \label{prop-dicho-urs-EP}
Let $\H \in \urs(G)$ that comes from an extremely proximal action. Let $\K \in \urs(G)$ such that $\H \preccurlyeq \K$ and $\H \neq \K$. Then $\env(\H) \preccurlyeq \K$.
\end{prop}

\begin{proof}
Let $Z$ be a compact $G$-space that is minimal and extremely proximal and such that $\mathcal{S}_G(Z) = \H$. Fix $K \in \K$, and assume that $K$ does not act minimally on $Z$. According to Proposition \ref{prop-fix-pt-min} this implies that the URS $\K$ fixes a point in $Z$, i.e.\ $\K \preccurlyeq \H$. Since moreover $\H, \K$ satisfy $\H \preccurlyeq \K$ by assumption, we deduce that $\H = \K$, which is a contradiction. Therefore $K$ acts minimally on $Z$. Since moreover there exists $H \in \H$ such that $H \leq K$, we are in position to apply Lemma \ref{lem-urs-acts-min}, from which the conclusion follows.
\end{proof}

It should be noted that Proposition \ref{prop-dicho-urs-EP} is false without the extreme proximality assumption, as in general there are plenty of URS's between $\H$ and $\env(\H)$.

\begin{lem} \label{lem-env-act-min-ep}
Let $\H \in \urs(G)$ that comes from an extremely proximal action. Then $\env(\H)$ acts minimally on $\H$. 
\end{lem}

\begin{proof}
Let $Z$ be a compact $G$-space that is minimal and extremely proximal and such that $\mathcal{S}_G(Z) = \H$, and let $N = \env(\H)$. Without loss of generality we may assume that $\H$ is not a point, since otherwise there is nothing to prove. This ensures that $N$ acts non-trivially on $Z$. By extreme proximality $N$ must act minimally on $Z$ (see Lemma \ref{lem-normal-minimal}), and therefore also on $\H$ by Proposition \ref{prop-mini-urs-min}.
\end{proof}

\begin{rmq}
The extreme proximality assumption cannot be removed in Lemma \ref{lem-env-act-min-ep}. Indeed it is not true in general that, given $\H \in \urs(G)$, $\H$ remains a URS of $\env(\H)$. Indeed, as explained in \cite{Gla-Wei}, any minimal subshift on two letters gives rise to a URS $\H$ of the lamplighter group $G = C_2 \wr \mathbb{Z}$, such that $\H$ is contained in the Chabauty space $\sub(L)$ of the base group $L = \oplus C_2$. In particular $\env(\H)$ lies inside the abelian group $L$, and it follows that $\env(\H)$ acts trivially on $\H$.
\end{rmq}

\begin{prop} \label{prop-action-urs-same-urs}
Let $\H \in \urs(G)$ that comes from an extremely proximal action, and assume $\H$ is not a point. Then:
\begin{enumerate}[label=(\alph*)]
\item \label{item-sam-urs} The action of $G$ on $\H$ gives rise to the same URS, i.e.\ $\mathcal{S}_G(\H) = \H$.
\item \label{item-sam-urs-faith} If moreover $\H$ comes from a faithful extremely proximal action, then the action of $G$ on $\H$ is faithful.
\end{enumerate}
\end{prop}

\begin{proof}
Write $\K = \mathcal{S}_G(\H)$. By definition we have $\H \preccurlyeq \K$. Argue by contradiction and suppose $\H \neq \K$. Then applying Proposition \ref{prop-dicho-urs-EP}, we deduce that $\env(\H)$ acts trivially on $\H$. But $\env(\H)$ also acts minimally on $\H$ by Lemma \ref{lem-env-act-min-ep}, so we deduce that $\H$ must be a point, a contradiction. This shows \ref{item-sam-urs}. 

For \ref{item-sam-urs-faith}, arguing as in the proof of Lemma \ref{lem-env-act-min-ep} we see that any non-trivial normal subgroup $N$ of $G$ acts minimally on $\H$. Since $\H$ is not a point, we have in particular that $N$ acts non-trivially on $\H$.
\end{proof}

\begin{rmq} \label{rmq-Z-metric}
Proposition \ref{prop-action-urs-same-urs} implies that, as far as our interest lies inside the URS associated to a minimal and extremely proximal action (and not the space $Z$ itself), there is no loss of generality in assuming that $(G,Z)$ is a sub-system of $(G,\sub(G))$. See also Remark \ref{rmq-reform-comes-from-EP}.
\end{rmq}

\subsection{Amenable URS's}

Recall that we say that $\H \in \urs(G)$ is amenable if every $H \in \H$ is amenable. The following lemma already appeared in \cite[Prop.\ 2.21]{LBMB}.

\begin{lem} \label{lem-am-urs-small-bnd}
If $\H \in \urs(G)$ is amenable and $X$ is a $G$-boundary, then $\H \preccurlyeq \mathcal{S}_G(X)$.
\end{lem}

\begin{proof}
Since $\H$ is amenable, $\H$ must fix a point in the compact $G$-space $\prob(X)$. Now $X$ is the unique minimal $G$-invariant subspace of $\prob(X)$ since $X$ is a $G$-boundary, so by Lemma \ref{lem-fp-Xmin} we have that $\H$ fixes a point in $X$, i.e.\ $\H \preccurlyeq \mathcal{S}_G(X)$.
\end{proof}

\begin{prop} \label{prop-compar-urs-bnd}
Let $X$ be a compact minimal $G$-space such that $\H = \mathcal{S}_G(X)$ is amenable, and let $Y$ be a $G$-boundary such that $X$ and $Y$ are disjoint. Then $\env(\H)$ acts trivially on $Y$.

In particular if $\env(\H)$ is co-amenable in $G$, a non-trivial $G$-boundary is never disjoint with $X$.
\end{prop}

\begin{proof}
The fact that $\env(\H)$ must act trivially on $Y$ follows by applying Lemma \ref{lem-am-urs-small-bnd} and Proposition \ref{prop-same-urs}. Since an amenable group has no non-trivial boundary, the second statement follows.
\end{proof}

Proposition \ref{prop-compar-urs-bnd} says that when $G$ admits an amenable URS whose envelope is co-amenable, a non-trivial $G$-boundary is never disjoint with $X$. This conclusion is not satisfactory for our concerns as it depends on the choice of a space $X$ and not only on $G$. Although there is no hope to get a better conclusion in full generality, the next result, which will play an important role in Section \ref{sec-lattices-urs}, will remove this dependence under an extreme proximality assumption.

We recall from the introduction that we say that $G$ is \textbf{boundary indivisible} if two non-trivial $G$-boundaries are never disjoint. 

\begin{prop} \label{prop-A_G-EP}
Assume that $G$ admits an amenable $\H \in \urs(G)$ that comes from an extremely proximal action, and let $X$ be a non-trivial $G$-boundary.
\begin{enumerate}[label=(\alph*)]
\item \label{item-prop-A-EP} Either $\mathcal{S}_G(X) = \H$, or $\env(\H)$ acts trivially on $X$.
\item \label{item-prop-A-EP-2} Assume that $\env(\H)$ is co-amenable in $G$. Then $\mathcal{S}_G(X) = \H$, and $G$ is boundary indivisible.
\end{enumerate}
\end{prop}

\begin{proof}
\ref{item-prop-A-EP}. Since $\H$ is amenable, we have $\H \preccurlyeq \mathcal{S}_G(X)$ by Lemma \ref{lem-am-urs-small-bnd}. Now if we assume $\H \neq \mathcal{S}_G(X)$, then according to Proposition \ref{prop-dicho-urs-EP} we have $\env(\H) \preccurlyeq \mathcal{S}_G(X)$, which exactly means that $\env(\H)$ acts trivially on $X$.

\ref{item-prop-A-EP-2}. If $\mathcal{S}_G(X) \neq \H$ then the action of $G$ on $X$ factors through an action of $G/\env(\H)$ by \ref{item-prop-A-EP}. But by assumption the latter is amenable, so has no non-trivial boundaries. So it follows that $X$ is trivial, a contradiction. Therefore all non-trivial $G$-boundaries have the same stabilizer URS $\H$. Since moreover $\H$ cannot be a point (because otherwise $G$ would be amenable), the fact that $G$ is boundary indivisible follows from Proposition \ref{prop-same-urs}.
\end{proof}

For a countable group $\Gamma$, the \textbf{Furstenberg URS} of $\Gamma$ is the stabilizer URS associated to the action of $\Gamma$ on its Furstenberg boundary. We refer to \cite{LBMB} for the proof of the following properties.

\begin{prop} \label{prop-background-A_G}
Let $\Gamma$ be a countable group, and $\A_\Gamma$ its Furstenberg URS. Then the following hold:
\begin{enumerate}[label=(\alph*)]
\item $\A_\Gamma$ is amenable, and $\H \preccurlyeq \A_\Gamma$ for every amenable $\H \in \urs(\Gamma)$.
\item If $X$ is a $\Gamma$-boundary, then $\A_\Gamma \preccurlyeq \mathcal{S}_\Gamma(X)$. If moreover there is $x \in X$ such that $\Gamma_x$ is amenable, then $\A_\Gamma = \mathcal{S}_\Gamma(X)$.
\item $\A_\Gamma$ is invariant under $\mathrm{Aut}(\Gamma)$.
\end{enumerate}
\end{prop}

\begin{prop} \label{prop-furst-urs-inside-env}
Let $\Gamma$ be a countable group, and let $\Lambda = \env(\A_\Gamma)$ be the envelope of the Furstenberg URS of $\Gamma$. Then $\Lambda$ acts minimally on $\A_\Gamma$, and $\A_\Gamma = \A_{\Lambda}$.
\end{prop}

\begin{proof}
The conjugation action of $\Gamma$ on the normal subgroup $\Lambda = \env(\A_\Gamma)$ induces a map $\Gamma \rightarrow \mathrm{Aut}(\Lambda)$. Since $\A_{\Lambda}$ is invariant under $\mathrm{Aut}(\Lambda)$ by Proposition \ref{prop-background-A_G}, it is in particular $\Gamma$-invariant. Moreover the action of $\Gamma$ on $\A_{\Lambda}$ is clearly minimal since it is already the case for $\Lambda$. Therefore $\A_{\Lambda}$ is an amenable URS of $\Gamma$, so it follows that $\A_{\Lambda} \preccurlyeq \A_\Gamma$ since $\A_\Gamma$ is larger than any amenable URS of $\Gamma$. On the other hand $\A_\Gamma$ is a closed and $\Lambda$-invariant subset of $\sub(\Lambda)$ consisting of amenable subgroups, so by the domination property applied to $\A_{\Lambda}$ we must have $\A_\Gamma \preccurlyeq \A_{\Lambda}$. Equality follows.
\end{proof}

\begin{rmq} \label{rmq-reform-comes-from-EP}
When $\envA$ is co-amenable in $\Gamma$, the fact that $\A_\Gamma$ comes from a faithful and extremely proximal action is equivalent to saying that the $\Gamma$-action on $\A_\Gamma$ is faithful and extremely proximal. The direct implication is consequence of Proposition \ref{prop-action-urs-same-urs}, and the converse follows from Proposition \ref{prop-A_G-EP}. This gives us an intrinsic reformulation of the assumption of Theorem \ref{thm-intro-URS-lattice-prod} inside the Chabauty space of $\Gamma$.
\end{rmq}

\section{Extremely proximal actions} \label{sec-epntf}

If $\X$ is a Hausdorff $\Gamma$-space and $U \subset \X$, we denote by $\Gamma_U$ the set of elements of $\Gamma$ acting trivially on $\X \setminus U$. We say that the action of $\Gamma$ on $\X$ is \textbf{micro-supported} if $\Gamma_U$ is non-trivial for every non-empty open set $U$. 

We will need the following easy lemma.

\begin{lem} \label{lem-mic-supp-not-solv}
Assume that the action of $\Gamma$ on $\X$ is micro-supported, and let $U$ be a non-empty open set. Then $\Gamma_U$ is not solvable.
\end{lem}

\begin{proof}
Assume that $\Lambda$ is a subgroup of $\Gamma_U$ whose action on $U$ is micro-supported, and let $V$ be a non-empty open subset of $U$. By assumption there exists a non-trivial $\lambda_1 \in \Lambda_V$, so that we may find an open set $W \subset V$ such that $W$ and $\lambda_1(W)$ are disjoint. For $\lambda_2 \in \Lambda_W$, the commutator $[\lambda_1, \lambda_2]$ coincides with $\lambda_2^{-1}$ on $W$, and is therefore non-trivial provided that $\lambda_2$ is non-trivial. It follows by induction that if $\Gamma_{U,n}$ is the $n$-th term of the derived series of $\Gamma_U$, then the action of $\Gamma_{U,n}$ on $U$ is micro-supported. In particular $\Gamma_{U,n}$ is never trivial, and $\Gamma_U$ is not solvable.
\end{proof}

In this section we will consider the following setting:

\medskip

\textbf{(EP)} \textit{$\Gamma$ is a discrete group, $Z$ is a compact $\Gamma$-space, and the action of $\Gamma$ on $Z$ is faithful, minimal and extremely proximal. In order to avoid trivialities, we assume that $Z$ has at least three points.}

\medskip

Unless specified otherwise, in the remaining of this section $\Gamma$ and $Z$ will be assumed to satisfy (EP). Our goal is to derive various properties on the group $\Gamma$ that will be used in later sections.

\begin{lem} \label{lem-normal-minimal}
Let $N \leq \mathrm{Homeo}(Z)$ be a non-trivial subgroup that is normalized by $\Gamma$. Then $N$ acts minimally and does not fix any probability measure on $Z$.
\end{lem}

\begin{proof}
Assume there exists $C \subsetneq Z$ that is closed and $N$-invariant. Since $C$ is compressible and $N$ is normalized by $\Gamma$, wee see that $N$ has a fixed point in $Z$. Now the set of $N$-fixed points is $\Gamma$-invariant, so it has to be the entire $Z$ by minimality, and $N$ is trivial. The same argument shows the absence of $N$-invariant probability measure on $Z$, since an extremely proximal action is also strongly proximal by Theorem \ref{thm-ep-sp}.
\end{proof}

In all this section the terminology \textit{topologically free} (see Definition \ref{def-stab-urs}) has to be understood with $\Gamma$ viewed as a discrete group. Therefore that the action is not topologically free means that there exists $\gamma \neq 1$ which acts trivially on a non-empty open subset of $Z$.

\begin{lem} \label{lem-ep-ms}
If the action of $\Gamma$ on $Z$ is not topologically free, then it is micro-supported.
\end{lem}

\begin{proof}
Let $U$ be a non-empty open subset of $Z$. Let $\gamma$ be a non-trivial element such that there is a non-empty open set $V$ on which $\gamma$ acts trivially, and let $g \in \Gamma$ such that $g(X \setminus V) \subset U$. Then the non-trivial element $g \gamma g^{-1}$ acts trivially outside $U$, so $\Gamma_U$ is non-trivial.
\end{proof}

\begin{defi}
Let $\Gamma^0$ be the subgroup of $\Gamma$ generated by the elements $\gamma \in \Gamma$ such that $\fix(\gamma)$ has non-empty interior. 
\end{defi}

\begin{rmq}
When $\Gamma$ is a countable group, $\Gamma^0$ is also equal to the envelope of the URS $\mathcal{S}_\Gamma(Z)$ by Lemma \ref{lem-equiv-fix-pt}.
\end{rmq}

Recall that the \textbf{monolith} $\mon(\Gamma)$ is the intersection of all non-trivial normal subgroups of $\Gamma$. We say $\Gamma$ is \textbf{monolithic} if $\mon(\Gamma)$ is non-trivial.

\begin{prop} \label{prop-core-deriv0}
Assume that the action of $\Gamma$ on $Z$ is not topologically free. Then the following hold:
\begin{enumerate}[label=(\alph*)]
	\item \label{item-comm-E} The commutators $[\gamma_1,\gamma_2]$, where $\fix(\gamma_1) \cap \fix(\gamma_2)$ has non-empty interior, generate $[\Gamma^0,\Gamma^0]$.
	\item \label{item-E-mon} $\Gamma$ is monolithic, and one has $\mon(\Gamma) = [\Gamma^0,\Gamma^0]$.
	\item \label{item-N-centr-triv} Any non-trivial normal subgroup of $\Gamma$ has trivial centralizer.
	\item \label{item-M-simple} If the action of $[\Gamma^0,\Gamma^0]$ on $Z$ is extremely proximal, then $[\Gamma^0,\Gamma^0]$ is a simple group.
	\item \label{item-Gamma-virt-s} $\Gamma$ is virtually simple if and only if $\Gamma^0$ has finite index in $\Gamma$ and finite abelianization.
\end{enumerate}
\end{prop}

\begin{proof}
\ref{item-comm-E} Denote by $N$ the subgroup generated by the set of $[\gamma_1,\gamma_2]$, where $\gamma_1,\gamma_2$ act trivially on a common open set. We show that for every $g,h$ fixing non-empty open sets $U,V$, the commutator $[g,h]$ belongs to $N$. Since $\Gamma^0$ is generated by all these elements $g,h$, this will show that $\Gamma^0 / N$ is abelian. Hence $[\Gamma^0,\Gamma^0] \leq N$, and the other inclusion is clear. 

First note that $N$ is not trivial by Lemmas \ref{lem-mic-supp-not-solv} and \ref{lem-ep-ms}. Therefore $N$ acts minimally on $Z$ according to Lemma \ref{lem-normal-minimal}, and we may find $s \in N$ such that the open set $W = U \cap s(V)$ is non-empty. Since $g$ and $h^s$ fix $W$ by construction, we have $[g,h^s] \in N$. But since $s \in N$, we deduce that $[g,h] = [h,s]^g [g,h^s][s,h]$ is a product of three elements of $N$, and hence belongs to $N$, as desired.

\ref{item-E-mon} We shall show that any non-trivial normal subgroup $N$ contains $[\Gamma^0,\Gamma^0]$. Since $[\Gamma^0,\Gamma^0]$ is itself a non-trivial normal subgroup, this will prove that it is the monolith of $\Gamma$. By a classical commutator manipulation (see e.g.\ Lemma 4.1 from \cite{Nek-fp}), there exists an open set $U$ such that $N$ contains the derived subgroup of $\Gamma_U$. Now let $\gamma_1, \gamma_2$ fixing an open set $V$. If $\gamma$ is such that $\gamma(Z \setminus V) \subset U$, then $\gamma_1^\gamma,\gamma_2^\gamma$ are supported inside $U$, so that $[\gamma_1^\gamma,\gamma_2^\gamma] = [\gamma_1, \gamma_2]^\gamma$ is contained in $N$. Since $N$ is normal, $[\gamma_1, \gamma_2] \in N$. Now all these elements generate $[\Gamma^0,\Gamma^0]$ by \ref{item-comm-E}, hence the conclusion. 

\ref{item-N-centr-triv} If $N$ is a normal subgroup of $\Gamma$, then so is $C_\Gamma(N)$. Therefore they cannot be both non-trivial, because otherwise the intersection would be abelian and would contain $[\Gamma^0,\Gamma^0]$ by the previous paragraph, a contradiction.

\ref{item-M-simple} If the action of $[\Gamma^0,\Gamma^0]$ on $Z$ is extremely proximal, then according to \ref{item-E-mon} the monolith $N$ of $[\Gamma^0,\Gamma^0]$ is non-trivial. Since $N$ is characteristic in $[\Gamma^0,\Gamma^0]$, $N$ is normal in $\Gamma$, and hence contains $[\Gamma^0,\Gamma^0]$ by \ref{item-E-mon}. So $N = [\Gamma^0,\Gamma^0]$ and $[\Gamma^0,\Gamma^0]$ is simple.

\ref{item-Gamma-virt-s} For $\Gamma$ to be virtually simple it is clearly necessary that the normal subgroup $[\Gamma^0,\Gamma^0]$ has finite index in $\Gamma$. Conversely, if this condition holds then the action of $[\Gamma^0,\Gamma^0]$ on $Z$ is extremely proximal (Lemma \ref{lem-EP-to-cocompact}), and $[\Gamma^0,\Gamma^0]$ is simple by \ref{item-M-simple}.
\end{proof}

\begin{defi}
Let $\X$ be a topological space, and let $\Gamma$ be a group acting on $\X$. A non-empty open set $\Omega \subset \X$ is \textbf{wandering} for $\Gamma$ if the translates $\gamma(\Omega)$, $\gamma \in \Gamma$, are pairwise disjoint. We say that $\Omega$ is wandering for $\gamma$ if it is wandering for $\left\langle \gamma\right\rangle$.
\end{defi}

\begin{prop} \label{prop-free-wander}
Let $\Gamma$ and $Z$ as in (EP). Then there exist an open set $\Omega$ and a non-abelian free subgroup $\mathbb{F}_2 = \Delta \leq \Gamma$ such that $\Omega$ is wandering for $\Delta$. 
\end{prop}

\begin{proof}
Following Glasner \cite{Glas-top-dyn-group}, we consider pairwise disjoint non-empty open sets $U_-,U_+,V_-,V_+$, and elements $a,b \in \Gamma$ such that $a(Z \setminus U_-) \subset U_+$ and $b(Z \setminus V_-) \subset V_+$. Let $W = U_- \cup U_+ \cup V_- \cup V_+$. It follows from a ping-pong argument than any non-trivial reduced word in the letters $a,b$ sends the complement of $W$ inside $W$, so that the subgroup $\Delta$ generated by $a,b$ is free \cite[Th.\ 3.4]{Glas-top-dyn-group}.

Upon reducing $W$ if necessary, we may find an open set $\Omega$ such that $\Omega \cap W = \emptyset$ and $a(\Omega) \subset U_{+}$, $a^{-1}(\Omega) \subset U_{-}$, $b(\Omega) \subset V_{+}$ and $b^{-1}(\Omega) \subset V_{-}$. Induction on the word length shows that if the left-most letter of $\gamma \in \Delta$ is respectively $a,a^{-1},b,b^{-1}$, then $\gamma(\Omega)$ lies respectively inside $U_+,U_-,V_+,V_-$. In particular $\Omega \cap \gamma(\Omega)$ is empty since $\Omega$ is disjoint from $W$, so $\Omega$ is wandering for $\Delta$.
\end{proof}

\begin{prop} \label{prop-embeds-wreath}
Retain notations (EP). Then the wreath product $\Gamma_U \wr \mathbb{F}_2$ embeds into $\Gamma$ for every open subset $U \subsetneq Z$ that is not dense.
\end{prop}

\begin{proof}
Let $\Omega$ and $\Delta$ as in Proposition \ref{prop-free-wander}, and let $\Lambda$ be the subgroup of $\Gamma$ generated by $\Gamma_\Omega$ and $\Delta$. Since $\Omega$ is wandering for $\Delta$, all the conjugates $\gamma \Gamma_\Omega \gamma^{-1}$ pairwise commute, and it follows that $\Lambda$ is isomorphic to $\Gamma_\Omega \wr \mathbb{F}_2$. Now if $U$ is an in the statement, by extreme proximality the group $\Gamma_U$ is isomorphic to a subgroup of $\Gamma_\Omega$, hence the conclusion.
\end{proof}

The argument in the following proof is borrowed from \cite{Capr-proceed}.

\begin{prop} \label{prop-lin-rep-core}
In the setting (EP), if the action of $\Gamma$ on $Z$ is not topologically free, then $\Gamma$ cannot have any faithful linear representation.
\end{prop}

\begin{proof}
Let $U$ be an open subset of $Z$. By Lemmas \ref{lem-mic-supp-not-solv} and \ref{lem-ep-ms}, we may find a non-abelian finitely generated subgroup $B$ inside $\Gamma_U$. Now if we choose $U$ small enough, it follows from Proposition \ref{prop-embeds-wreath} that the finitely generated group $\Lambda = B \wr \mathbb{F}_2$ is isomorphic to a subgroup of $\Gamma$. Since $\Lambda$ is not residually finite \cite{Gruen57}, it admits no faithful linear representation by Malcev's theorem \cite{Malcev-lin-RF}, and a fortiori the same is true for $\Gamma$.
\end{proof}

Recall that a subgroup $\Lambda$ of a group $\Gamma$ is \textbf{commensurated} if all conjugates of $\Lambda$ are commensurable, where two subgroups are commensurable if the intersection has finite index in both.

The beginning of the argument in the proof of the following proposition already appeared in \cite{LB-Wes}. The idea is to extend classical techniques for normal subgroups to certain commensurated subgroups.

\begin{prop} \label{prop-commens-wander}
Let $\Gamma$ and $Z$ as in (EP), and assume that the action of $\Gamma$ on $Z$ is not topologically free. If $\Lambda$ is a commensurated subgroup of $\Gamma$ such that there exists an element of $\Lambda$ admitting a wandering open set, then $\Lambda$ contains the monolith $\mon(\Gamma)$.
\end{prop}

\begin{proof}
Let $\lambda \in \Lambda$ admitting a wandering open set $\Omega$. We shall first prove that $[\Gamma_{\Omega},\Gamma_{\Omega}]$ is contained in $\Lambda$. Let $g,h \in \Gamma_{\Omega}$, and let also $n \geq 1$. Since $\Omega$ is wandering we have $\Omega \cap \lambda^n (\Omega) = \emptyset$. It follows that the commutator $[g,\lambda^n]$ is trivial outside $\Omega \cup \lambda^n (\Omega)$, and coincides with $g$ on $\Omega$ and with $\lambda^n g^{-1} \lambda^{-n}$ on $\lambda^n (\Omega)$. Therefore its commutator with $h$ is trivial outside $\Omega$, and is coincides with $[g,h]$ on $\Omega$. But since $g,h \in \Gamma_{\Omega}$, the elements $[[g,\lambda^n],h]$ and $[g,h]$ actually coincide everywhere, i.e.\ $[[g,\lambda^n],h] = [g,h]$.

Now since $\Lambda$ is commensurated in $\Gamma$, there exists $n_0 \geq 1$ such that $[[g,\lambda^{n_0}],h]$ belongs to $\Lambda$. Applying the previous argument with $n = n_0$, we deduce that $[g,h]$ belongs to $\Lambda$. 

In order to prove the statement, it is enough to prove that $[\Gamma_{C},\Gamma_{C}]$ is contained in $\Lambda$ for every closed subset $C \subsetneq Z$ according to Proposition \ref{prop-core-deriv0}. So let $C$ a proper closed subset of $Z$. By minimality and extreme proximality, there is $\gamma \in \Gamma$ such that $\gamma(C) \subset \Omega$. Fix such a $\gamma$, and choose some integer $n_1 \geq 1$ such that $\gamma^{-1} \lambda^{n_1} \gamma$ belongs to $\Lambda$. Set $\lambda' = \gamma^{-1} \lambda^{n_1} \gamma$ and $\Omega' = \gamma^{-1}(\Omega)$. This $\Omega'$ is wandering for $\lambda'$ and $\lambda' \in \Lambda$, so $[\Gamma_{\Omega'},\Gamma_{\Omega'}]$ is contained in $\Lambda$ by the first paragraph. Since $C \subset \Omega'$, we have $\Gamma_{C} \leq \Gamma_{\Omega'}$, and the proof is complete.
\end{proof}

\begin{prop} \label{prop-commens-dicho}
Let $\Gamma$ and $Z$ as in (EP), and assume that the action of $\Gamma$ on $Z$ is not topologically free. Then there exists a free subgroup $\mathbb{F}_2 \leq \Gamma$ such that for every commensurated subgroup $\Lambda \leq \Gamma$ not containing the monolith $\mon(\Gamma)$, we have $\mathbb{F}_2 \cap \Lambda = 1$.
\end{prop}

\begin{proof}
Let $\mathbb{F}_2$ be a free subgroup of $\Gamma$ as in the conclusion of Proposition \ref{prop-free-wander}. If $\Lambda \leq \Gamma$ is a commensurated subgroup such that $\mathbb{F}_2 \cap \Lambda \neq 1$, then in particular $\Lambda$ contains an element admitting a wandering open set. So by Proposition \ref{prop-commens-wander} we have $\mon(\Gamma) \leq \Lambda$. This shows that every commensurated subgroup not containing $\mon(\Gamma)$ intersects $\mathbb{F}_2$ trivially. 
\end{proof}

\begin{cor} \label{cor-epntf-embed-amen}
Let $\Gamma$ and $Z$ as in (EP), and assume that the action of $\Gamma$ on $Z$ is not topologically free. If $G$ is a locally compact amenable group whose connected component $G^0$ is a Lie group, then there exists no injective homomorphism $\Gamma \rightarrow G$.
\end{cor}

\begin{proof}
Argue by contradiction and assume $\Gamma$ embeds in $G$. Let $U$ be an open subgroup of $G$ containing $G^0$ as a cocompact subgroup (the existence of $U$ follows from van Dantzig's theorem \cite[Th.\ 7.7]{HR} applied to the totally disconnected group $G / G^0$). Such a $U$ is commensurated in $G$, so the subgroup $\Gamma \cap U$ is commensurated in $\Gamma$. If there exists $U$ such that $\Gamma \cap U$ does not contain $\mon(\Gamma)$, then according to Proposition \ref{prop-commens-dicho} we may find a non-abelian free subgroup $\mathbb{F}_2 \leq \Gamma$ such that $\mathbb{F}_2 \cap U = 1$. In particular $G$ contains the non-amenable group $\mathbb{F}_2$ as a discrete subgroup, which contradicts amenability of $G$. Therefore $\mon(\Gamma)$ is contained in $U$ for every choice of $U$. Since compact open subgroups form a basis at $1$ in $G / G^0$ by van Dantzig's theorem, it follows that $\mon(\Gamma)$ actually lies inside $G^0$. Now since $G^0$ is a connected Lie group, the group $\mathrm{Aut}(G^0)$ is linear, so the map $G \rightarrow \mathrm{Aut}(G^0)$ induced by the conjugation action of $G$ on $G^0$ is not injective in restriction to $\Gamma$ by Proposition \ref{prop-lin-rep-core}. Therefore this map must vanish on $\mon(\Gamma)$, which means that $\mon(\Gamma)$ actually lies inside the center of $G^0$. In particular $\mon(\Gamma)$ is abelian, which contradicts Proposition \ref{prop-core-deriv0}.
\end{proof}

\section{The proofs of Theorems \ref{thm-intro-URS-lattice-prod} and \ref{thmintro-normal-coamenable-lc}} \label{sec-lattices-urs}

\subsection{Boundary-minimal subgroups} \label{subsec-normal-bound}

In this paragraph we consider the following property:

\begin{defi} \label{def-bnd-min}
Let $G$ be a topological group, and $L$ a closed subgroup of $G$. We say that $L$ is \textbf{boundary-minimal} if there exists a non-trivial $G$-boundary on which $L$ acts minimally.
\end{defi}

It should be noted that being boundary-minimal does not prevent $L$ from being amenable. For instance the action of Thompson's group $T$ on the circle $\mathbb{S}^1$ is a boundary action, and the abelian subgroup of $T$ consisting of rotations acts minimally on $\mathbb{S}^1$. Other examples may be found among the groups acting on trees considered in \S \ref{subsec-background-g(f,f')}, where the stabilizer of a vertex is an amenable subgroup acting minimally on the ends of the tree.

In the sequel we will mainly focus on the case when $L$ is normal in $G$, or more generally when $L$ belongs to a URS (see Proposition \ref{prop-no-disj-wcoam}). By contrast with the previous examples, a normal boundary-minimal subgroup is never amenable, as a normal amenable subgroup of $G$ acts trivially on any $G$-boundary. Recall that Furman showed \cite[Prop.\ 7]{Furman-min-strg} (see also Caprace--Monod \cite[Prop.\ 3]{rel-amen}) that if $N$ is non-amenable normal subgroup of a locally compact group $G$,  there always exists a $G$-boundary on which $N$ acts non-trivially. It turns out that we actually have the following:

\begin{thm} \label{thm-normal-nonam-Gbnd}
Let $G$ be a locally compact group, and $N$ a closed normal non-amenable subgroup. Then there exists a non-trivial $G$-boundary on which the $N$-action is minimal and strongly proximal.
\end{thm}

Before going into the proof, we introduce the following terminology:

\begin{defi} 
Let $G$ be a topological group, and $H$ a closed subgroup of $G$. We say that a $G$-boundary $X$ is a \textbf{$(G,H)$-boundary} if the $H$-action on $X$ is strongly proximal.
\end{defi}

The usual argument for the existence of a universal $G$-boundary extends to this setting:

\begin{lem} \label{lem-relbnd-univ}
Let $G$ be a topological group, and $H$ a closed subgroup of $G$. Then there exists a unique $(G,H)$-boundary $\partial_{sp}(G,H)$ such that for any $(G,H)$-boundary $X$, there is a $G$-equivariant surjection $\partial_{sp}(G,H) \to X$. 

Moreover if $H=N$ is normal in $G$, then $N$ acts minimally on $\partial_{sp}(G,N)$.
\end{lem}

\begin{proof}
We consider the product $P$ of all $(G,H)$-boundaries. The $H$-action on $P$ remains strongly proximal. The group $G$ admits a unique minimal closed invariant subset $M$ in $P$, and the $H$-action on $M$ is strongly proximal. Hence $M$ satisfies the desired property. This shows existence, and the argument for uniqueness is the same as in the case $H=G$ \cite[II.4.1]{Gl-PF}.

By strong proximality, the subgroup $H$ has a unique minimal invariant subspace $Y$ in $\partial_{sp}(G,H)$. So when $H$ is a normal subgroup of $G$, we have that $Y$ must be $G$-invariant, and hence $Y = \partial_{sp}(G,H)$ by minimality of $G$.
\end{proof}

As in the classical case $H = G$, the universal property of the space $\partial_{sp}(G,H)$ allows to obtain the following:

\begin{lem} \label{lem-rel-bnd-fi}
Let $G$ be a topological group, and $N \leq L \leq G$ closed subgroups such that $N$ is normal in $G$ and $L$ has finite index in $G$. Then the $L$-action on $\partial_{sp}(L,N)$ extends to a $G$-action, and $\partial_{sp}(L,N)$ and $\partial_{sp}(G,N)$ are isomorphic $G$-spaces.
\end{lem}

\begin{proof}
The argument is exactly the same as in the usual case, so we refer the reader to \cite[II.4.3-4]{Gl-PF}. The proof uses the universal property from \ref{lem-relbnd-univ} and makes use of the existence of a finite index subgroup $L'$ of $L$ that is normal in $G$. The only thing that needs to be observed here is that $L'$ may be chosen to contain $N$ and that the action by conjugation by elements of $G$ on $L'$ indeed preserves $N$, because we have chosen $N$ to be normal in $G$.
\end{proof}

We now give the proof of Theorem \ref{thm-normal-nonam-Gbnd}:

\begin{proof}
Since $N$ is normal in $G$, the desired conclusion is equivalent to the existence of a non-trivial $(G,N)$-boundary. Upon replacing $G$ by the quotient $G / \mathrm{Rad}(G)$ of $G$ by its amenable radical, and $N$ by its image closure in $G / \mathrm{Rad}(G)$, we may assume without loss of generality that the group $G$ has trivial amenable radical. According to \cite[Th. 3.3.3]{Bur-Mon-bound-coho-rigid} (see also the more recent reference \cite{Burg-LC}), $G$ admits a finite index characteristic open subgroup $G'$ that is isomorphic to a direct product $S \times G_{td}$, where $S = \prod S_i$ is a direct product of (center-free, non-compact) connected simple Lie groups, and $G_{td}$ is a totally disconnected subgroup. Note that we may very well replace $N$ by $N \cap G'$, and since the desired conclusion is invariant by passing to a finite index open subgroup according to Lemma \ref{lem-rel-bnd-fi}, we may actually assume that $G = G' = S \times G_{td}$.

If the normal subgroup $N$ has a non-trivial projection to $S$, then there is a factor $S_i$ on which $N$ maps densely, and it follows that the $G$-action on $\partial_{sp} S_i$ is a non-trivial $(G,N)$-boundary. Hence we may assume that $N$ lies inside $G_{td}$, i.e.\ we are reduced to the case where $G$ is a totally disconnected group.

In this situation, by considering the preimage of a compact open subgroup in the quotient $G/N$, we find an open subgroup $O \leq G$ that is commensurated in $G$ and that contains $N$ as a cocompact subgroup. Since $O$ is commensurated in $G$, the uniqueness of $\partial_{sp} O$ and its invariance under taking finite index closed subgroups of $O$ allow to extend the action of $O$ on $X = \partial_{sp} O$ to an action of $G$ (see \cite[Th. 6.2]{Dai-Glasner}), and this action is continuous since $O$ is open in $G$. Note that $O$ is not amenable since $N$ is not amenable, so $X$ is non-trivial. Moreover the $N$-action on $X$ is minimal and strongly proximal since $N$ is cocompact and normal in $O$ \cite[II.3.1-2]{Gl-PF}, so it follows that $X$ is a non-trivial $(G,N)$-boundary. 
\end{proof}

\subsection{Weakly co-amenable subgroups} \label{subsec-pf-thm-lc}

In this paragraph we consider the following weakening of the notion of co-amenability. 

\begin{defi} \label{def-w-co-am}
Let $G$ be a topological group, and $H$ a subgroup of $G$. We say that $H$ is \textbf{weakly co-amenable} in $G$ if whenever $Q$ is a non-trivial convex compact $G$-space in which $H$ fixes a point, $Q$ is not irreducible.
\end{defi}

The following properties readily follow from the definition.

\begin{prop} \label{prop-wco-elem}
Let $K \leq H \leq G$ be subgroups of $G$.
\begin{enumerate}[label=(\roman*)]
\item \label{item-co-wc} If $H \leq G$ is co-amenable then $H$ is weakly co-amenable.
\item \label{item-wco-n} For a normal subgroup $N \lhd G$, weakly co-amenable is equivalent to co-amenable.
\item \label{item-am-wc} If $H \leq G$ is amenable and weakly co-amenable in $G$, then $G$ is amenable. 
\item \label{item-quo-wc} If $\varphi: G \rightarrow G'$ is continuous with dense image and $H \leq G$ is weakly co-amenable, then $\varphi(H)$ is weakly co-amenable in $G'$.
\item \label{item-sb-wc} If $K$ is weakly co-amenable in $G$, then $H$ is weakly co-amenable in $G$.
\item \label{item-trans-wc} If $K$ is co-amenable in $H$ and $H$ is weakly co-amenable in $G$, then $K$ is weakly co-amenable in $G$.
\end{enumerate}
\end{prop}

\begin{proof}
\ref{item-co-wc} If $Q$ is non-trivial convex compact $G$-space with $H$-fixed points, then there is a $G$-fixed point by co-amenability of $H$ in $G$, so $Q$ is not irreducible.

\ref{item-wco-n} If $N \lhd G$ is not co-amenable, there is a convex $Q$ such that $\mathrm{Fix}(N)$ is non-empty but $\mathrm{Fix}(G)$ is empty. Since $N$ is normal $\mathrm{Fix}(N)$ is $G$-invariant, so that by Zorn's lemma $\mathrm{Fix}(N)$ contains an irreducible convex $G$-space, which is non-trivial since $\mathrm{Fix}(G)$ is empty. This shows $N$ is not weakly co-amenable.

The proofs of \ref{item-am-wc}, \ref{item-quo-wc}, \ref{item-sb-wc} and \ref{item-trans-wc} are similar verifications, and we leave them to the reader.
\end{proof}

\begin{rmq}
As for co-amenability, it is natural to wonder whether weak co-amenability of $K$ in $G$ implies weak co-amenability of $K$ in $H$. In view of \ref{item-wco-n}, the same counter-examples given in \cite{MonPop-coam} show that the answer is negative in general.
\end{rmq}

By the correspondence between irreducible convex compact $G$-spaces and $G$-boundaries, weak co-amenability admits the following characterization:

\begin{prop} \label{prop-reform-w-coam}
A subgroup $H \leq G$ is weakly co-amenable in $G$ if and only if for every non-trivial $G$-boundary $X$, there is no probability measure on $X$ that is fixed by $H$.
\end{prop}

\begin{proof}
Follows from Theorem \ref{thm-conv-bnd}.
\end{proof}

The following shows how weak co-amenability naturally appears for boundary indivisible groups (see also Proposition \ref{prop-am-wcoam}).

\begin{prop} \label{prop-no-disj-wcoam}
Let $G$ be a boundary indivisible locally compact group, and $L$ a closed subgroup of $G$ that is boundary-minimal and uniformly recurrent. Then $L$ is weakly co-amenable in $G$.
\end{prop}

\begin{proof}
Write $\H$ for the closure of $L^G$ in $\sub(G)$, which is a URS by assumption. Let $X$ be a non-trivial $G$-boundary on which $L$ acts minimally, and let $Y$ be a $G$-boundary on which $L$ fixes a probability measure. We have to show that $Y$ is trivial. Since $\H$ fixes a point in $\prob(Y)$ and the $G$-action on $\prob(Y)$ is strongly proximal by Theorem \ref{thm-conv-bnd}, $\H$ fixes a point in $Y$ by Lemma \ref{lem-fp-Xmin}. So there exists $y \in Y$ such that $L \leq G_y$, and it follows that $G_y$ acts minimally on $X$. Therefore by Lemma \ref{lem-triv-disj} $X$ and $Y$ are disjoint, and since $X$ is non-trivial and $G$ is boundary indivisible, this is possible only if $Y$ is trivial.
\end{proof}

\subsection{The proof of Theorem \ref{thmintro-normal-coamenable-lc}}

In this paragraph we shall give the proof of Theorem \ref{thmintro-normal-coamenable-lc} from the introduction. We will make use of the following result.

\begin{prop}[Furstenberg] \label{prop-furst}
Let $G$ be a locally compact group, $H \leq G$ a closed subgroup of finite covolume, and $X$ a $G$-boundary. Then $X$ is a $H$-boundary.
\end{prop}

For completeness we repeat the argument from \cite[Prop.\ 4.4]{Furst-bourb}.

\begin{proof}
Write $Q = \prob(X)$, and consider a closed $H$-invariant subspace $Q' \subseteq Q$. We have to show that $X \subseteq Q'$. The set \[ \X = \left\{(gH,\mu) \, : \, \mu \in g(Q')\right\} \subseteq G/H \times Q \] is a well-defined, closed, $G$-invariant subspace of $G/H \times Q$. Fix a $G$-invariant probability measure $m_{G/H}$ on $G/H$, and consider \[ Y = \left\{ \nu \in \prob(\X) \, : \, p_{G/H}^*(\nu) = m_{G/H} \right\},\] where $p_{G/H}$ is the projection from $G/H \times Q$ onto the first factor, and $p_{G/H}^*$ is the induced push-forward operator. Then $Y$ is a closed (and hence compact) $G$-invariant subspace of $\prob(\X)$, and $p_Q^*: Y \rightarrow \prob(Q)$ is continuous. So $p_Q^*(Y)$ is closed in $\prob(Q)$, and by strong proximality of the $G$-action on $Q$ (Theorem \ref{thm-conv-bnd}), $p_Q^*(Y)$ must intersect $Q$. Now $X$ being the unique minimal closed $G$-invariant subspace of $Q$, one has $X \subseteq p_Q^*(Y)$. For every $x \in X$, we therefore have $\nu_x \in \prob(\X)$ such that $p_{G/H}^*(\nu_x) = m_{G/H}$ and $p_{Q}^*(\nu_x) = \delta_x$. This implies $m_{G/H} \left\{g H \, : \, x \in g(Q')\right\} = 1$ for every $x$, and it easily follows that $X \subseteq Q'$.
\end{proof}

\begin{rmq}
In the case when $H$ is cocompact in $G$, strong proximality of the action of $H$ on $X$ also follows from \cite[II.3.1]{Gl-PF} applied to the action on $\prob(X)$; and minimality follows \cite[IV.5.1]{Gl-PF} from disjointness of the $G$-spaces $G/H$ and $X$ \cite[III.6.1]{Gl-PF}.
\end{rmq}

\begin{thm} \label{thm-G-no-disj-bnd}
Assume that $H$ admits an amenable URS $\H$ that comes from an extremely proximal action, and such that $\env(\H)$ is co-amenable in $H$. Let $G$ be a locally compact group containing $H$ as a closed subgroup of finite covolume. Then:
\begin{enumerate}[label=(\alph*)]
\item \label{item-case-G} $G$ is boundary indivisible.
\item \label{item-case-sequence} More generally if $L$ is a locally compact group such that there is a sequence a topological group homomorphisms $G = G_0 \rightarrow G_1 \rightarrow \ldots \rightarrow G_n = L$ such that either $G_i \rightarrow G_{i+1}$ has dense image, or $G_i \rightarrow G_{i+1}$ is an embedding of $G_i$ as a closed subgroup of finite covolume in $G_{i+1}$; then $L$ is boundary indivisible.
\end{enumerate}

In particular whenever $G$ maps continuously and with dense image to a product $G_1 \times G_2$, one factor $G_i$ must be amenable.
\end{thm}

\begin{proof}
Since $\H$ is amenable, $\H$ comes from an extremely proximal action, and $\env(\H)$ is co-amenable in $H$, the group $H$ is boundary indivisible by Proposition \ref{prop-A_G-EP}. Now by Proposition \ref{prop-furst}, the property of being boundary indivisible is inherited from closed subgroups of finite covolume. Indeed if $X,Y$ are disjoint $G$-boundaries, i.e.\ $X \times Y$ is a $G$-boundary, then $X \times Y$ is also a boundary for $H$ by Proposition \ref{prop-furst}, hence of $X$ or $Y$ must be trivial since $H$ is boundary indivisible. This shows \ref{item-case-G}. Since boundary indivisibility passes to dense continuous images, and is inherited from closed subgroups of finite covolume, \ref{item-case-sequence} follows from \ref{item-case-G}.

Finally if $G_1,G_2$ are as in the last statement and $X_i = \partial_{sp} G_i$, then $X_1 \times X_2$ is a boundary for $G_1 \times G_2$, which is boundary indivisible by the previous paragraph. So one factor $X_i$ must be trivial, which exactly means that $G_i$ is amenable by Theorem \ref{thm-amean-bnd}.
\end{proof}

\begin{rmq}
In the proof of Theorem \ref{thm-G-no-disj-bnd} we obtain that $G$ is boundary indivisible from the same property for $H$, which is itself deduced from Proposition \ref{prop-A_G-EP} (which in turn relies notably on Proposition \ref{prop-same-urs}). We note that the order in which the argument is developed seems to matter, in the sense that the arguments applied to $H$ do not seem to be applicable directly to the group $G$. Indeed we do not know whether a group $G$ as in Theorem \ref{thm-G-no-disj-bnd} falls into the setting of Proposition \ref{prop-same-urs}, i.e.\ we do not know whether all non-trivial $G$-boundaries have the same stabilizer URS. We actually believe this might be false in general.
\end{rmq}

We note at this point that the proof of Theorem \ref{thmintro-normal-coamenable-lc} from the introduction is now complete. Indeed the fact that a group $G$ as in Theorem \ref{thmintro-normal-coamenable-lc} is boundary indivisible is Theorem \ref{thm-G-no-disj-bnd}. Statement \ref{item-intr-urs-wco} follows from Proposition \ref{prop-no-disj-wcoam}, and statement \ref{item-intr-N-am-coam} follows from \ref{item-intr-urs-wco} together with Theorem \ref{thm-normal-nonam-Gbnd}.

The following remark explains a comment from the introduction.

\begin{rmq} \label{rmq-urs-wco-not-co}
Theorem 1.1 from \cite{LB-irr-wreath} show that the countable groups $G(F,F')$ embed as lattices in a group $G$ of the form $G = N \rtimes \aut$, and under sertain assumptions on the permutation groups $F,F'$, all the assumptions of Theorem \ref{thmintro-normal-coamenable-lc} are satisfied (see Section \ref{sec-trees}). If $M$ is a cocompact subgroup of $\aut$ acting minimally on $\partial T_d$, then $L = N \rtimes M$ is a cocompact (hence uniformly recurrent) subgroup of $G$, and $L$ is boundary-minimal in $G$ since $\partial T_d$ is a $G$-boundary. However when $M$ is non-unimodular (e.g.\ if $M$ is a non-ascending HNN-extension of a profinite group $K$ over open subgroups $K_1,K_2$ such that $K_1$ and $K_2$ do not have the same index in $K$), then $M$ is not co-amenable in $\aut$, and $L$ is not co-amenable in $G$. This shows that the conclusion of statement \ref{item-intr-urs-wco} in Theorem \ref{thmintro-normal-coamenable-lc} that $L$ is weakly co-amenable in $G$ cannot be strengthened by saying that $L$ is co-amenable in $G$.
\end{rmq}

\subsection{The proof of Theorem \ref{thm-intro-URS-lattice-prod}}

Recall that if $G$ is a topological group, the \textbf{quasi-center} $\mathrm{QZ}(G)$ of $G$ is the subgroup of $G$ containing the elements $g \in G$ having an open centralizer. Note that $\mathrm{QZ}(G)$ contains the elements having a discrete conjugacy class, so in particular it contains all discrete normal subgroups. Recall also that the \textbf{elliptic radical} of $G$ is the largest normal subgroup of $G$ in which every compact subset generates a relatively compact subgroup. It is a closed characteristic subgroup of $G$.

The following is slightly more complete than Theorem \ref{thm-intro-URS-lattice-prod} from the introduction.

\begin{thm} \label{thm-URS-lattice-prod}
Let $\Gamma$ be a countable group whose Furstenberg URS $\A_\Gamma$ comes from a faithful and extremely proximal action, and assume that $\envA$ is co-amenable in $\Gamma$. Let $G$ be a locally compact group containing $\Gamma$ as a lattice. Consider the following properties:
\begin{enumerate}[label=(\alph*)]
\item \label{item-env-fg} $\envA$ is finitely generated;
\item \label{item-env-fi} $\envA$ has finite index in $\Gamma$, and $\Gamma$ admits a finitely generated subgroup with finite centralizer;
\item \label{item-mon-fi} $\envA$ has finite index in $\Gamma$ and $\envA$ has finite abelianization.
\end{enumerate}
Then:
\begin{itemize}
\item \ref{item-env-fg}, \ref{item-env-fi} both imply that $G$ cannot be a product of two non-compact groups.
\item \ref{item-mon-fi} implies that any continuous morphism with dense image from $G$ to a product of locally compact groups $G \rightarrow G_1 \times G_2$ is such that one factor $G_i$ is compact.
\end{itemize}
\end{thm}

\begin{proof}
Of course we may assume that $\envA$ is non-trivial, since otherwise there is nothing to prove. According to Proposition \ref{prop-core-deriv0} we have in particular that $\Gamma$ is monolithic, and $\mon(\Gamma) = [\envA,\envA]$. For simplicity in all the proof we write $E = \envA$ and $M = \mon(\Gamma) = [E,E]$.

Assume that $\varphi: G \rightarrow G_1 \times G_2$ is continuous with dense image, and denote by $p_i$ the projection $G_1 \times G_2 \rightarrow G_i$, $i=1,2$. We will show that one factor must be compact. Upon modding out by the maximal compact normal subgroup of the identity component $G^0$, which intersects $\Gamma$ trivially since $\Gamma$ has no non-trivial finite normal subgroup (Proposition \ref{prop-core-deriv0}), we may also assume that $G^0$ has no non-trivial compact normal subgroup. This implies in particular that $G^0$ is a connected Lie group \cite{MontZip}.

By the assumption that $E$ is co-amenable in $\Gamma$, we can apply Theorem \ref{thm-G-no-disj-bnd}, which says that one factor, say $G_2$, must be amenable. We then apply Corollary \ref{cor-epntf-embed-amen}, which tells us that the map $p_2 \circ \varphi$ is not injective in restriction to $\Gamma$. By definition of $M$ we deduce that $M \leq \varphi^{-1}(G_1 \times 1)$.

Assume now that \ref{item-mon-fi} holds. Then $M$, being of finite index in $\Gamma$, is a lattice in $G$, and is contained in the closed normal subgroup $\varphi^{-1}(G_1 \times 1)$. Therefore we deduce that $\varphi^{-1}(G_1 \times 1)$ is cocompact in $G$, and that $p_2 \circ \varphi(G)$ is a compact subgroup of $G_2$. Since $p_2 \circ \varphi(G)$ is also dense in $G_2$, we have that $G_2$ is compact.

We now have to deal with \ref{item-env-fg}, \ref{item-env-fi}, in which case $\varphi$ is the identity and $G = G_1 \times G_2$. Without loss of generality, we may assume that the projections $p_i(\Gamma)$ are dense. The proofs of the two cases will share a common mechanism, given by the following easy fact:

\begin{lem} \label{lem-reduc-fac-cp}
If there exists a subgroup $L \leq G$ whose centralizer $C_G(L)$ contains $G_2$, $C_G(L)$ is open in $G$, and $C_G(L) \cap \Gamma$ is finite, then $G_2$ is compact.
\end{lem}

Indeed, since $\Gamma$ must intersect an open subgroup $O \leq G$ along a lattice of $O$, it follows that $C_G(L)$ is compact, and a fortiori so is $G_2$. 

We start with case \ref{item-env-fg}. Note that $\Gamma$ cannot embed in the group $\mathrm{Aut}(G_1^0)$ (Proposition \ref{prop-lin-rep-core}), and hence the subgroup $M$ centralizes $G_1^0$. The subgroup $M \cap G_1^0$ is normal in $\Gamma$ and cannot be equal to $M$ (otherwise $M$ would be abelian), so $M \cap G_1^0 = 1$. Consider $H_1 = \overline{p_1(E)}$, which is normal in $G_1$ by density of $p_1(\Gamma)$. Note that $H_1$ is compactly generated in view of the assumption that $E$ is finitely generated. Since $M=[E,E]$, the group $H_1 / M$ is abelian, and therefore of the form $\mathbb{Z}^n \times \mathbb{R}^m \times C$ for some compact group $C$ \cite[Th.\ 9.8]{HR}. It follows that the group $Q_1 = H_1 /H_1^0$ admits a discrete cocompact normal subgroup $\Delta$, which is an extension of $M$ by a free abelian group. Being characteristically simple and non-amenable, the group $M$ has trivial elliptic radical, so the group $\Delta$ also has trivial elliptic radical. Now since $Q_1$ is compactly generated, there is a compact open normal subgroup $K$ of $Q_1$ such that $K \rtimes \Delta$ has finite index in $Q_1$ (see e.g.\ \cite[Lem. 4.4]{BCGM-lat-am}), so we deduce that $Q_1$ has a compact open elliptic radical. Since any connected group has compact elliptic radical \cite{MontZip}, we deduce that $H_1$ has a compact elliptic radical $R$, and $H_1 /R$ is discrete-by-connected. The compact group $R$ is also normal in $G$, and therefore we can mod out by $R$ and assume that $R$ is trivial, so that $H_1^0$ is open in $H_1$. Since $H_1^0$ centralizes $M$, any $\gamma \in E$ such that $p_1(\gamma)$ belongs to $H_1^0$ centralizes $M$, and therefore is trivial by Proposition \ref{prop-core-deriv0}. Therefore $H_1^0$ is open in $H_1$ and intersects the dense subgroup $p_1(E)$ trivially, so it follows that $H_1^0$ is trivial, and $p_1(E)$ is a discrete subgroup of $G$. 

Observe that $p_1(E)$ is centralized by $G_2$ and normalized by $G_1$, and hence is normal in $G$. Being a discrete normal subgroup of $G$, $p_1(E)$ therefore lies in the quasi-center $\mathrm{QZ}(G)$. Since $p_1(E)$ is finitely generated, the centralizer of $p_1(E)$ in $G$ is actually open in $G$. Moreover the subgroup $\Gamma \cap C_G(p_1(E))$ is normal in $\Gamma$ since $C_G(p_1(E))$ is normal in $G$, but clearly does not contain $M$, and hence is trivial by Proposition \ref{prop-core-deriv0}. Therefore we can apply Lemma \ref{lem-reduc-fac-cp} with $L = p_1(E)$, and we obtain the conclusion.

We now deal with \ref{item-env-fi}. Let $Z$ be a minimal compact $\Gamma$-space on which the $\Gamma$-action is faithful and extremely proximal and such that $\mathcal{S}_\Gamma(Z) = \A_\Gamma$. By Proposition \ref{prop-furst} (actually an easy case of it) the action of $E$ on $Z$ is also minimal, and it is extremely proximal by Lemma \ref{lem-EP-to-cocompact}. Moreover the associated stabilizer URS remains equal to $\A_\Gamma$, and is also the Furstenberg URS of $E$ by Proposition \ref{prop-furst-urs-inside-env}. So $E$ satisfies all the assumptions of case \ref{item-env-fi} of the theorem, so it is enough to prove the result under the additional assumption $\Gamma = E$. In this case we have $M = \left[\Gamma, \Gamma\right]$ thanks to Proposition \ref{prop-core-deriv0}, so it follows that $p_2(\Gamma)$ is abelian. By density of the projection the group $G_2$ is also abelian, and hence $G_2$ lies in the center of $G$. Therefore $\Gamma$ is normalized by the dense subgroup $\Gamma G_2$, and it follows that $\Gamma$ is normal in $G$. In particular $\Gamma \leq \mathrm{QZ}(G)$, and the conclusion follows by applying Lemma \ref{lem-reduc-fac-cp} with $L$ a f.g.\ subgroup of $\Gamma$ such that $C_\Gamma(L)$ is finite.
\end{proof}

\begin{rmq}
The conclusions of the theorem also hold for any group $H$ that is commensurable with $G$ up to compact kernels (two groups $G_1,G_2$ are commensurable up to compact kernels if there exist $K_i \leq H_i \leq G_i$ such that $H_i$ is open and of finite index in $G_i$, $K_i$ is a compact normal subgroup of $H_i$, and $H_1 / K_1$ and $H_2 / K_2$ are isomorphic). The arguments follow the same lines, by observing that if $H$ is a finite index open subgroup of $G$ and $K$ a compact normal subgroup of $H$, then the subgroup $\Lambda = \Gamma \cap H$ (which, as a finite index subgroup of $\Gamma$, necessarily contains $M$) embeds as a lattice in $H / K$. We leave the details to the reader.
\end{rmq}

\section{Groups acting on trees} \label{sec-trees}

\subsection{Amenable URS's and groups acting on trees}

In this paragraph $T$ is a locally finite tree, and $H$ acts continuously on $T$ by isometries. The assumption that $T$ is locally finite is not essential here, and the results admit appropriate generalizations for non-locally finite trees (using the compactification from \cite[Prop.\ 4.2]{Mon-Sha-cocy}).

Recall that the $H$-action on $T$ is \textbf{minimal} if there is no proper invariant subtree, and of \textbf{general type} if $H$ has no finite orbit in $T \cup \partial T$. The following is well-known, and essentially goes back to Tits \cite{Tits70} (see also \cite{Pa-Va} and Proposition \ref{prop-folk-EP} for details).

\begin{prop} \label{prop-tree-ep}
If the action of $H \leq \mathrm{Aut}(T)$ is minimal and of general type, then the action of $H$ on $\partial T$ is minimal and extremely proximal.
\end{prop}

Theorem \ref{thm-G-no-disj-bnd} therefore implies the following result:

\begin{cor} \label{cor-tree-lc}
Let $H \leq \mathrm{Aut}(T)$ be a locally compact group whose action on $T$ is continuous, minimal and of general type. Assume that end stabilizers are amenable, and the envelope of $\mathcal{S}_H(\partial T)$ is co-amenable in $H$. Assume $H$ embeds as a subgroup of finite covolume in $G$. Then whenever $G$ maps continuously and with dense image to a product $G_1 \times G_2$, one factor $G_i$ must be amenable.
\end{cor}

The conclusion of Corollary \ref{cor-tree-lc} implies in particular that whenever $H$ embeds in $G$ with finite covolume, then $G$ cannot be a product of two non-amenable groups. The following example, which is largely inspired from \cite[Ex.\ II.8]{CaMo-decomp}, shows that the group $G$ can nonetheless be a product of two non-compact groups. 

\begin{ex} \label{ex-sl2-prod-non-cpt}
Let $\mathbf{k} = \mathbb{F}_p(\!(t)\!)$ be the field of Laurent series over the finite field $\mathbb{F}_p$, and let $\alpha \in \Aut(\mathbf{k})$ be a non-trivial automorphism of $\mathbf{k}$. The group $L = \mathrm{SL}(2,\mathbf{k})$ acts on a $(p+1)$-regular tree, $2$-transitively and with amenable stabilizers on the boundary. This action extends to a continuous action of $H = L \rtimes_{\alpha} \mathbb{Z}$, so that $H$ satisfies all the assumptions of Corollary \ref{cor-tree-lc}. Nevertheless $H$ embeds diagonally in the product $G = (L \rtimes \Aut(\mathbf{k})) \times \mathbb{Z}$ as a closed subgroup of finite covolume since $G/H$ is compact and $H$ and $G$ are unimodular.
\end{ex}

We will need the following fact. If $A$ is a subtree of $T$, by the fixator of $A$ we mean the subgroup fixing pointwise $A$. 

\begin{prop} \label{prop-env-A-tree}
Let $\Gamma \leq \mathrm{Aut}(T)$ be a countable group whose action on $T$ is minimal and of general type, and such that end-stabilizers in $\Gamma$ are amenable. Then $\A_\Gamma = \mathcal{S}_\Gamma(\partial T)$, and $\envA$ is the subgroup generated by fixators of half-trees.
\end{prop}

\begin{proof}
Since the $\Gamma$-action on $\partial T$ is extremely proximal, it is also strongly proximal by Theorem \ref{thm-ep-sp}. So $\partial T$ is a $\Gamma$-boundary with amenable stabilizers, and we deduce that $\A_\Gamma = \mathcal{S}_\Gamma(\partial T)$ by Proposition \ref{prop-background-A_G}.

Now according to Lemma \ref{lem-equiv-fix-pt}, the subgroup $\env(\mathcal{S}_\Gamma(\partial T)) = \envA$ is generated by the elements $\gamma \in \Gamma$ whose fixed point set in $\partial T$ has non-empty interior. Since half-trees form a basis of the topology in $\partial T$, the statement follows.
\end{proof}

Before going to the proof of Corollary \ref{cor-intro-tree-disc}, we make the following observation:

\begin{rmq} \label{rmq-virt-spl}
For $\Gamma$ acting on $T$ (action minimal and general type) such that the action on $\partial T$ is not topologically free, virtual simplicity of $\Gamma$ is equivalent to $\Gamma^0$ being of finite index in $\Gamma$ and $\Gamma^0$ having finite abelianization, where $\Gamma^0$ is the subgroup generated by fixators of half-trees. See statement \ref{item-Gamma-virt-s} of Proposition \ref{prop-core-deriv0}.
\end{rmq}

\begin{proof}[Proof of Corollary \ref{cor-intro-tree-disc}]
In view of Proposition \ref{prop-env-A-tree}, the assumptions on $\Gamma$ imply that the Furstenberg URS of $\Gamma$ comes from a faithful extremely proximal action. The fact that end-stabilizers are all non-trivial means that the action of $\Gamma$ on $\partial T$ is not topologically free, and by the above observation virtual simplicity of $\Gamma$ is equivalent to $\envA$ being of finite index in $\Gamma$ and with finite abelianization. The first statement of the corollary therefore follows from Theorem \ref{thm-URS-lattice-prod}, case \ref{item-mon-fi}, and the second statement from Theorem \ref{thmintro-normal-coamenable-lc}.
\end{proof}

\subsection{Groups with prescribed local action} \label{subsec-background-g(f,f')}

In the next paragraphs we will illustrate the results of the previous sections on a family of groups acting on trees, which contains instances of discrete and non-discrete groups. The purpose of this paragraph is to recall the definition and give a brief description of known properties of these groups.

We will denote by $\Omega$ a set of cardinality $d \geq 3$ and by $T_d$ a $d$-regular tree. The vertex set and edge set of $T_d$ will be denoted respectively $V_d$ and $E_d$. We fix a coloring $c: E_d \rightarrow \Omega$ such that neighbouring edges have different colors. For every $g \in \aut$ and every $v \in V_d$, the action of $g$ on the star around $v$ gives rise to a permutation of $\Omega$, denoted $\sigma(g,v)$, and called the local permutation of $g$ at $v$. These permutations satisfy the identity \begin{equation} \label{eq:loc-perm} \sigma(gh,v) = \sigma(g,hv) \sigma(h,v) \end{equation} for every $g,h \in \aut$ and $v \in V_d$. 


Given a permutation group $F \leq \Sy$, the group $U(F)$ introduced by Burger and Mozes in \cite{BM-IHES} is the group of automorphisms $g \in \aut$ such that $\sigma(g,v) \in F$ for all $v$. It is a closed cocompact subgroup of $\aut$.

\begin{defi}
Given $F \leq F'\leq \Sy$, we denote by $G(F,F')$ the group of automorphisms $g \in \aut$ such that $\sigma(g,v) \in F'$ for all $v$ and $\sigma(g,v) \in F$ for all but finitely many $v$.
\end{defi}

That $G(F,F')$ is indeed a subgroup of $\aut$ follows from (\ref{eq:loc-perm}), and we note that we have $U(F) \leq G(F,F') \leq U(F')$. We make the following observation for future reference.

\begin{rmq} \label{rmq-indep}
As it follows from the definition, any element $\gamma \in G(F,F')$ fixing an edge $e$ can be (uniquely) written as $\gamma = \gamma_1 \gamma_2$, where each $\gamma_i$ belongs to $G(F,F')$ and fixes one of the two half-trees defined by $e$.
\end{rmq}

We recall that a permutation group is \textbf{semi-regular} if it acts freely, and \textbf{regular} if it acts freely and transitively. In the sequel we always assume that $F'$ preserves the $F$-orbits in $\Omega$ (this property ensures that the local action is indeed isomorphic to $F'$, see \cite[Lem.\ 3.3]{LB-ae}). The groups $G(F,F')$ satisfy the following properties (see \cite{LB-ae}):

\begin{enumerate}
	\item The group $G(F,F')$ is dense in the locally compact group $U(F')$. In particular $G(F,\Sy)$ is a dense subgroup of $\aut$.
	\item $G(F,F')$ admits a locally compact group topology (defined by requiring that the inclusion of $U(F)$ is continuous and open), and the action of $G(F,F')$ on $T_d$ is continuous but not proper as soon as $F \neq F'$. Endowed with this topology, the group $G(F,F')$ is compactly generated.
	\item stabilizers of vertices and stabilizers of ends in $G(F,F')$ are respectively locally elliptic and (locally elliptic)-by-cyclic. In particular they are amenable.
	\item $G(F,F')$ is a discrete group if and only if $F$ is semi-regular. When this is so, the group $G(F,F')$ is therefore a finitely generated group, and stabilizers of vertices and stabilizers of ends in $G(F,F')$ are respectively locally finite and (locally finite)-by-cyclic.	
\end{enumerate}

When $F$ is semi-regular and $F \neq F'$, the groups $G(F,F')$ are instances of groups obtained from a more general construction described in \cite[Sec.\ 4]{LB-cs} (more precisely, a variation of it), which provides discrete groups with a continuous Furstenberg URS (the later being the stabilizer URS associated to the action on the boundary of the tree on which these groups act). In the particular case of the groups $G(F,F')$, the Furstenberg URS can be explicitly described, see Proposition 4.28 and Corollary 4.29 in \cite{LBMB}.

\medskip

\textit{In the sequel whenever we use letters $F$ and $F'$, we will always mean that $F,F'$ are permutation groups on a set $\Omega$, that $F'$ contains $F$ and preserves the $F$-orbits in $\Omega$. Following \cite{Tits70}, we will denote by $G(F,F')^+$ the subgroup of $G(F,F')$ generated by fixators of edges, and by $G(F,F')^\ast$ the subgroup of index two in $G(F,F')$ preserving the bipartition of $T_d$. 
}

\medskip

The following result, also obtained in \cite[Prop.\ 9.16]{CRWes-approx}, supplements simplicity results obtained in \cite{LB-ae}, where the index of the simple subgroup was found explicitly under appropriate assumptions on the permutation groups.

\begin{prop} \label{prop-g-virt-simpl}
The group $G(F,F')$ has a simple subgroup of finite index if and only if $F$ is transitive and $F'$ is generated by its point stabilizers.
\end{prop}

\begin{proof}
These conditions are necessary by \cite[Prop.\ 4.7]{LB-ae}. Conversely, assume $F$ transitive and $F'$ generated by its point stabilizers. By \cite[Prop.\ 4.7]{LB-ae} again, $G(F,F')^+$ has index two in $G(F,F')$, so in particular it is compactly generated. If $M$ is the monolith of $G(F,F')$, which is simple and open in $G(F,F')$ by \cite[Cor.\ 4.9]{LB-ae}, we have to show $M$ has finite index. According to Remark \ref{rmq-indep}, $G(F,F')^+$ is also the subgroup generated by fixators of half-trees, and therefore by Proposition \ref{prop-core-deriv0} $M$ is the commutator subgroup of $G(F,F')^+$. The abelianization of $G(F,F')^+$ is therefore a finitely generated abelian group, which is generated by torsion elements since $G(F,F')^+$ is generated by locally elliptic subgroups (fixators of edges). Therefore this abelianization is finite, and it follows that $M$ has finite index in $G(F,F')$.
\end{proof}

\subsection{Boundaries of $G(F,F')$} \label{subsec-bnd-g(f,f')}

In this paragraph we use results from the previous sections in order to study the boundaries of the discrete groups $G(F,F')$. The following result shows that several properties of the set of boundaries are governed by the permutation groups, and that rigidity phenomena occur under mild conditions of the permutation groups.

\begin{thm} \label{thm-boundaries-G(F,F')}
Assume that $F$ is semi-regular, $F \neq F'$, and write $\Gamma = G(F,F')$. The following are equivalent: 
\begin{enumerate}[label=(\roman*)]
\item \label{item-for-F'} The subgroup of $F'$ generated by its point stabilizers has at most two orbits in $\Omega$.
\item \label{item-env-egal} $\Gamma / \envA$ is isomorphic to one of $C_2$, $D_\infty$ or $D_\infty \rtimes C_2$.
\item \label{item-quo-env-moy} $\envA$ is co-amenable in $\Gamma$.
\item \label{item-same-A} $\mathcal{S}_\Gamma(X) = \mathcal{A}_\Gamma$ for every non-trivial $\Gamma$-boundary $X$.
\item \label{item-bnd-dsjt} $\Gamma$ is boundary indivisible.
\end{enumerate}
\end{thm}

We will need preliminary results before proving Theorem \ref{thm-boundaries-G(F,F')}.

\begin{lem} \label{lem-env-G(F,F')^+}
Assume that $F$ is semi-regular. The envelope of the Furstenberg URS of $G(F,F')$ is equal to $G(F,F')^+$.
\end{lem}

\begin{proof}
Write $\Gamma = G(F,F')$ and $\Gamma^+ = G(F,F')^+$. According to Proposition \ref{prop-env-A-tree}, $\envA$ is the subgroup generated by fixators of half-trees in $\Gamma$. Therefore the inclusion $\envA \leq \Gamma^+$ is clear. The converse inclusion also holds true by Remark \ref{rmq-indep}, so equality follows.
\end{proof}

In view of Lemma \ref{lem-env-G(F,F')^+} and Proposition \ref{prop-A_G-EP}, we are led to consider the quotient $G(F,F') / G(F,F')^+$, and in particular study when it is amenable. To this end, we will denote by $F'^+$ the subgroup of $F'$ generated by its point stabilizers, and write $D = F'/F'^+$. Since $F'^+$ is normal in $F'$, we have an action of $F'$ on the set of orbits of $F'^+$, which factors through a free action of $D$.

\begin{prop} \label{prop-explicit-G/G^+}
The group $Q = G(F,F')^\ast / G(F,F')^+$ is isomorphic to the group $U(D)^\ast$, where $D = F'/F'^+$ is viewed as a semi-regular permutation group on the set of orbits of $F'^+$ in $\Omega$. If moreover $F$ is transitive, one has $Q = D \ast D$.
\end{prop}

\begin{proof}
We let $O_1,\ldots,O_r \subseteq \Omega$ be the orbits of $F'^+$, and we will freely identify the set of orbits with the integers $\left\{1,\ldots,r\right\}$. For every $a \in \Omega$, there is a unique $i \in \left\{1,\ldots,r\right\}$ such that $a \in O_{i}$, and we denote $i=i_a$.

We view the tree $T_d$ as the Cayley graph of the free Coxeter group of rank $d$, namely the group defined by generators $x_1,\ldots,x_d$ and relators $x_j^2 =1$ for all $j$. When adding relations of the form $x_a = x_b$ whenever $i_a = i_b$ (i.e.\ $a$ and $b$ are in the same $F'^+$-orbit), we obtain a free Coxeter group of rank $r$. It has a Cayley graph that is a regular tree of degree $r$, and we have a surjective map $p: T_d \rightarrow T_r$.

Two elements $v,w \in \ast_d C_2 = \left\langle x_1,\ldots,x_d \right\rangle $ have the same image in $\ast_r C_2$ if and only if one can write $w = v \prod w_j x_{a_j} x_{b_j} w_j^{-1}$ for some words $w_j$ and colors $a_j,b_j$ such that $i_{a_j} = i_{b_j}$. Since the inverse of $w_j$ is equal to the word obtained from $w_j$ by reversing the order, we have:

\begin{lem} \label{lem-same-fiber}
Two vertices $v,w$ of $T_d$ have the same projection in $T_r$ if and only if the distance between $v$ and $w$ is even, say $d(v,w) = 2m$, and if  $(a_1,\ldots,a_{2m})$ is the sequence of colors of the unique geodesic from $v$ to $w$, then the word $(i_{a_1},\ldots,i_{a_{2m}})$ is a concatenation of palindromes of even length.
\end{lem}

\begin{lem} \label{lem-local-perm-constant-modulo}
For every $g \in G(F,F')$ and every vertex $v$ on $T_d$, the image of $\sigma(g,v)$ in $D = F'/F'^+$ does not depend on $v$. We denote by $\sigma_g \in D$ the corresponding element, which is trivial when $g \in G(F,F')^+$.
\end{lem}

\begin{proof}
If $v,w$ are adjacent vertices and $a$ is the color of the edge between them, then $\sigma(g,v)(a) = \sigma(g,w)(a)$. So $\sigma(g,v) \sigma(g,w)^{-1} \in F'^+$. The first statement follows by connectedness. The fact that $\sigma_g$ is trivial on $G(F,F')^+$ is then clear because $g \mapsto \sigma_g$ is a morphism according to the first statement, which vanishes on fixators of edges.
\end{proof}

Note that the set of edges of $T_r$ inherits a natural coloring by the integers $1,\ldots,r$.

\begin{lem} \label{lem-action-new-tree}
There is a natural morphism $\varphi: G(F,F') \rightarrow \mathrm{Aut}(T_r)$ such that $\ker(\varphi) = G(F,F')^+$ and $\mathrm{Im}(\varphi) = U(D)$.
\end{lem}

\begin{proof}
We shall first define an action of $G(F,F')$ on the set of vertices of $T_r$. Let $g \in G(F,F')$. Let $v,w$ be two vertices of $T_d$, and $(a_1,\ldots,a_{n})$ be the sequence of colors from $v$ to $w$. If $v = v_1, \ldots, v_{n+1} = w$ are the vertices between $v$ and $w$, then the sequence of colors from $g(v)$ to $g(w)$ is $(\sigma(g,v_1)(a_1),\ldots,\sigma(g,v_n)(a_{n}))$. If $\sigma_g$ is the element defined in Lemma \ref{lem-local-perm-constant-modulo}, then one has $i_{\sigma(g,v_j)(a_j)} = \sigma_g(i_{a_j})$ for all $j=1,\ldots,n$. This shows in particular that if $v,w$ satisfy the condition of Lemma \ref{lem-same-fiber}, then the same holds for $g(v)$ and $g(w)$. This means that for every vertex $x$ of $T_r$, the formula \begin{equation} \label{eq:action-Tr} \varphi(g)(x) := p(g \tilde{x}), \end{equation} where $\tilde{x}$ is any vertex of $T_d$ such that $p(\tilde{x}) = x$, is a well-defined action of $G(F,F')$ on $T_r$. The fact that the tree structure is preserved is clear. Note that for every $g \in G(F,F')$, all the local permutations of $\varphi(g)$ are equal to $\sigma_g$: for every vertex $x$ of $T_r$, one has $\sigma(\varphi(g),x) = \sigma_g$. In particular the image of $\varphi$ lies inside $U(D)$.

We shall prove that $\ker(\varphi) = G(F,F')^+$. Let $g \in G(F,F')$ fixing an edge in $T_d$. Then $\varphi(g)$ also fixes an edge of $T_r$ by (\ref{eq:action-Tr}). Moreover one has $\sigma_g = 1$ (Lemma \ref{lem-local-perm-constant-modulo}), and it follows from the previous paragraph that all local permutations of $\varphi(g)$ are trivial. This implies $g \in \ker(\varphi)$. Conversely, we let $g$ be an element of $\ker(\varphi)$, and we prove that $g \in G(F,F')^+$. Note that since $\varphi(g)$ is trivial, one has $\sigma_g = 1$, i.e.\ all local permutations of $g$ are in $F'^+$. Now let $v$ be any vertex. By Lemma \ref{lem-same-fiber}, the sequence of colors $(a_1,\ldots,a_{2n})$ from $v$ to $g(v)$ gives rise to a sequence $(i_{a_1},\ldots,i_{a_{2n}})$ that is a concatenation of palindromes. For simplicity we treat the case where $(i_{a_1},\ldots,i_{a_{2n}})$ is a palindrome; the general case consists in repeating the argument for this case. Let $v = v_1, \ldots, v_{2n+1} = g(v)$ be the vertices between $v$ and $g(v)$ (note that $v_n$ is the midpoint between $v$ and $g(v)$). Since $(i_{a_1},\ldots,i_{a_{2n}})$ is a palindrome, one easily checks that there are elements $g_1, \ldots, g_n$ such that $g_j$ belongs to the stabilizer of $v_j$ in $G(F,F'^+)$ and $g' = g_1 \ldots g_n g$ fixes the vertex $v$ (this is obtained by successively folding the geodesic $[v,g(v)]$ onto itself starting from its midpoint in order to bring back $g(v)$ to $v$ with $g_1 \ldots g_n$). We now invoke the following easy fact, whose verification is left to the reader.

\begin{lem} \label{lem-plus-inversion}
Let $\gamma \in G(F,F')$ fixing a vertex $w$, and such that $\sigma(\gamma,w) \in F'^+$. Then $\gamma \in G(F,F')^+$.
\end{lem}

We apply Lemma \ref{lem-plus-inversion} to $g'$ and all $g_j$'s, and deduce that $g  = g_n^{-1} \ldots g_1^{-1} g'$ belongs to $G(F,F')^+$ as desired.

The last thing that remains to be proved in the statement of Lemma \ref{lem-action-new-tree} is that the image of $\varphi$ is equal to $U(D)$. The fact that $\varphi(g)$ always belongs to $U(D)$ has already been observed. For the converse inclusion, observe that since $G(F,F')$ acts transitively on the vertices of $T_r$ (as it is already the case on $T_d$), it is enough to check that the image of $\varphi$ contains $U(D)_x$ for some vertex $x$ of $T_r$. Now since $D$ acts freely on $\left\{1,\ldots,r\right\}$, the map $U(D)_x \rightarrow D$, $\gamma \mapsto \sigma(\gamma,x)$, is an isomorphism. Therefore it is enough to see that any action on the star around $x$ on $T_r$ can realized by an element of $G(F,F')$, and this is indeed the case (see e.g.\ Lemma 3.4 in \cite{LB-ae}).
\end{proof}

To finish the proof of the proposition, remark that the image of $G(F,F')^\ast$ by $\varphi$ is precisely $U(D)^\ast$. When $F$ is transitive then $D$ is also transitive, so that $U(D)^\ast$ has two orbits of vertices and one orbit of edges, and therefore splits as the free product $D \ast D$.
\end{proof}

\begin{rmq}
The case $F=F'$ is allowed in Proposition \ref{prop-explicit-G/G^+}, so that the conclusion also holds for the groups $U(F)$ from \cite{BM-IHES}.
\end{rmq}

Proposition \ref{prop-explicit-G/G^+} naturally leads us to isolate the following three situations. We keep the previous notation, so that $r$ is the number of orbits of $F'^+$ in $\Omega$, $D = F' / F'^+$ and $Q = G(F,F')^\ast / G(F,F')^+$:

\begin{enumerate}[leftmargin = 0.7cm]
\item $r=1$. In this case $T_r$ is a segment of length one, and $D$ and $Q$ are trivial.
\item \label{item-r=2} $r=2$. $T_r$ is a bi-infinite line, and this case splits into two disjoint sub-cases:
\begin{enumerate}[label=(\alph*)]
\item \label{item-F-intrans} If $F'$ is intransitive then $D$ is trivial, and $Q = U(1)^\ast = \mathbb{Z}$ (generated by a translation of $T_r$ of length $2$).
\item \label{item-F-trans} If $F'$ is transitive then we have $D=\mathrm{Sym}(2)$ and $Q = D \ast D = D_\infty$.
\end{enumerate}
\item $r \geq 3$. Then $Q = U(D)^\ast$ is a virtually free group (since it acts vertex transitively and with trivial edge stabilizers of $T_r$).
\end{enumerate}

Theorem \ref{thm-boundaries-G(F,F')} says that all properties stated there hold true if and only if $r \in \left\{1,2\right\}$. A sufficient condition for having $r=1$ is for instance that $F$ acts transitively on $\Omega$, $F \neq F'$ and $F'$ acts primitively, or quasi-primitively on $\Omega$. Recall that a permutation group is quasi-primitive if every non-trivial normal subgroup acts transitively.

But Theorem \ref{thm-boundaries-G(F,F')} also applies beyond the case of quasi-primitive permutation groups. For example a situation giving rise to case (\ref{item-r=2}) \ref{item-F-intrans} is when $F'$ has a fixed point and acts transitively on the complement. Examples giving rise to case (\ref{item-r=2}) \ref{item-F-trans} are for instance obtained by taking $F' = \mathrm{Sym}(n) \wr C_2 = \mathrm{Sym}(n) \wr \left\langle \tau \right\rangle$ acting naturally on $2n$ letters, and $F$ the subgroup generated by $((c_n,c_n),1)$ and $((1,1),\tau)$, where $c_n$ is a cycle of order $n$.

\begin{proof}[Proof of Theorem \ref{thm-boundaries-G(F,F')}]
\ref{item-for-F'} $\Rightarrow$ \ref{item-env-egal} follows from Lemma \ref{lem-env-G(F,F')^+} and Proposition \ref{prop-explicit-G/G^+} (and the discussion following its proof). \ref{item-env-egal} $\Rightarrow$ \ref{item-quo-env-moy} is clear. \ref{item-quo-env-moy} $\Rightarrow$ \ref{item-same-A} is Proposition \ref{prop-A_G-EP}. \ref{item-same-A} $\Rightarrow$ \ref{item-bnd-dsjt} is guaranteed by Proposition \ref{prop-same-urs} and the fact that $\mathcal{A}_\Gamma$ is not a point. Finally assume that \ref{item-for-F'} does not hold, i.e.\ $F'^+$ has at least three orbits in $\Omega$, and write $\Gamma^+ = G(F,F')^+$. By Proposition \ref{prop-explicit-G/G^+}, the group $Q = \Gamma / \Gamma^+$ has a subgroup of finite index that is free of rank at least 2. So there exist non-trivial $Q$-boundaries, and a fortiori these are non-trivial $\Gamma$-boundaries. If $X$ is such a boundary, then $\Gamma^+$ acts trivially on $X$. Since $\Gamma^+$ also acts minimally on $\partial T_d$, it follows that $X$ and $\partial T_d$ are disjoint $\Gamma$-boundaries, contradicting \ref{item-bnd-dsjt}. Therefore property \ref{item-bnd-dsjt} implies property \ref{item-for-F'}, and the proof is complete.
\end{proof}

\subsection{Weakly co-amenable subgroups} \label{subsec-coam-g(f,f')}

In this paragraph we show that subgroups of the groups $G(F,F')$ satisfy the following dichotomy:

\begin{prop} \label{prop-am-wcoam}
Assume that $F$ is regular and $F'$ is primitive. Then any subgroup of $G(F,F')$ is either (locally finite)-by-cyclic (and hence amenable) or weakly co-amenable.
\end{prop}

We will need the following lemma.

\begin{lem} \label{lem-prim-2urs-all}
Assume that $F'$ acts primitively on $\Omega$, and take two subgroups $H_1 \neq H_2$ in the Furstenberg URS of $G(F,F')$. Then $\left\langle H_1, H_2\right\rangle = G(F,F')^*$.
\end{lem}

\begin{proof}
Write $\Gamma = G(F,F')$. Recall from \cite[Prop.\ 4.28]{LBMB} that the Furstenberg URS of $\Gamma$ consists of subgroups $\Gamma_\xi^0$, $\xi \in \partial T_d$, where $\Gamma_\xi^0$ is the set of elements acting trivially on a neighbourhood of $\xi$. Given $\xi \neq \eta \in \partial T_d$, we show that the subgroup $\Lambda$ generated by $\Gamma_\xi^0$ and $\Gamma_\eta^0$ must be equal to $G(F,F')^*$. 

Take a vertex $v$ on the geodesic from $\xi$ to $\eta$, let $e_1,e_2$ be the edges containing $v$ and pointing towards $\xi$ and $\eta$, and $a,b$ the colors of $e_1,e_2$. Denote by $K(v)$ the subgroup of $\Gamma$ consisting of elements $\gamma$ fixing $v$ and such that $\sigma(\gamma,w) \in F$ for every $w \neq v$. Since $F'$ is primitive, $F'$ is generated by the point stabilizers $F_a'$ and $F_b'$. This implies that every element of $K(v)$ may be written as a product of elements fixing either the half-tree defined by $e_1$ containing $\xi$, or the half-tree defined by $e_2$ containing $\eta$, so that $K(v) \leq \Lambda$. Since $v$ was arbitrary, we also have $K(v') \leq \Lambda$ for $v'$ a neighbour of $v$ on the geodesic $[\xi,\eta]$. The conclusion now follows since for two neighbouring vertices $v,v'$, the subgroups $K(v),K(v')$ always generate $G(F,F')^*$ \cite[Cor.\ 3.10]{LB-ae}. 
\end{proof}

\begin{proof}[Proof of Proposition \ref{prop-am-wcoam}]
Write $\Gamma = G(F,F')$, and let $\Lambda$ be a subgroup of $\Gamma$ that is non-amenable, equivalently whose action of $T_d$ is of general type. By Proposition \ref{prop-reform-w-coam} we have to show that $\Lambda$ fixes no probability measure on any non-trivial $\Gamma$-boundary. Argue by contradiction and assume that $X$ is a non-trivial $\Gamma$-boundary on which $\Lambda$ fixes a probability measure $\mu$. According to Theorem \ref{thm-boundaries-G(F,F')}, we have $\mathcal{S}_\Gamma(X) = \mathcal{A}_\Gamma$. Therefore by Proposition \ref{prop-GW-1.2} there exist an almost 1-1 extension $\eta: \tilde{X} \rightarrow X$ and a factor map $\pi: \tilde{X} \rightarrow \mathcal{A}_\Gamma$. 

Let $Q \subset \prob(\tilde{X})$ be the set of $\nu$ such that $\eta^* \nu = \mu$, and write $R = \pi^*(Q)$, which is a closed $\Lambda$-invariant subset of $\prob(\mathcal{A}_\Gamma)$. Since the action of $\Lambda$ on $\partial T_d$ is strongly proximal and since $\A_\Gamma$ is a factor of $\partial T_d$ \cite[Prop.\ 2.10-4.28]{LBMB}, we deduce that $R$ contains some Dirac measures. Let $H \in \A_\Gamma$ such that there is $\nu \in \prob(\tilde{X})$ with $\eta^* \nu = \mu$ and $\pi^* \nu = \delta_H$. Such a measure $\nu$ must be supported in the set of $(x,H) \in \tilde{X}$, and it follows that $\mu$ is supported in the set of $H$-fixed points in $X$ (because $(x,H) \in \tilde{X}$ implies that $H \leq G_x$ by upper semi-continuity of the stabilizer map). But since $\Lambda$ does not fix any point in $\A_\Gamma$, we may find another $H' \in \A_\Gamma$ such that $\delta_{H'} \in R$, so that the same argument shows that $H'$ also acts trivially on the support of $\mu$. By Lemma \ref{lem-prim-2urs-all} the subgroups $H,H'$ generate $G(F,F')^*$, which is of index two in $\Gamma$. Therefore any point in the support of $\mu$ has a $\Gamma$-orbit of cardinality at most two, which is absurd since $\Gamma$ acts minimally on $X$ and $X$ is non-trivial by assumption.
\end{proof}

\begin{rmq}
Assume that $F$ is regular and $F'$ is primitive, and write $\Gamma = G(F,F')$. Let $\Lambda \leq \Gamma$ be a subgroup generated by two hyperbolic elements with sufficiently far apart axis. Then $\Lambda$ is not co-amenable in $\Gamma$ (see e.g.\ the argument in the proof of Theorem 2.4 in \cite{CaMo-Iso-disc}), but $\Lambda$ is weakly co-amenable in $\Gamma$ by Proposition \ref{prop-am-wcoam}.
\end{rmq}

\begin{rmq}
We mention that when $F'$ is primitive, following the proof of Corollary 4.14 from \cite{LBMB} (with minor modifications), one could prove that every non-trivial $G(F,F')$-boundary factors onto $\partial T_d$. This would provide an alternative proof of Proposition \ref{prop-am-wcoam}.
\end{rmq}

\subsection{Lattice embeddings of the groups $G(F,F')$} \label{subsec-g(f,f')-lat-em}

In this section we apply previous results of the article to the family of groups $G(F,F')$ and deduce some properties of general locally compact groups containing a group $G(F,F')$ as a lattice.

\begin{rmq}
As mentioned earlier, examples of lattice embeddings for the groups $G(F,F')$ are described in \cite{LB-irr-wreath}, and maybe it is worth pointing out that instances of lattice embeddings of these groups also appeared in \cite{LB-ae}. Indeed under appropriate assumptions on permutation groups $F \leq F',H \leq H'$, the inclusion of $G(F,F')$ in $G(H,H')$ has discrete and cocompact image \cite[Cor.\ 7.4]{LB-ae}.
\end{rmq}

\begin{cor} \label{cor-g(f,f')-lat-obstruc}
Assume that $F$ is regular, and that $F'$ is generated by its point stabilizers. Let $G$ be a locally compact group containing $G(F,F')$ as a lattice. Then the conclusions of Corollary \ref{cor-intro-tree-disc} hold.
\end{cor}

\begin{proof} 
The assumptions on $F,F'$ imply that $G(F,F')$ is virtually simple by Proposition \ref{prop-g-virt-simpl}, so Corollary \ref{cor-intro-tree-disc} applies.
\end{proof}

\begin{rmq}
In the setting of Corollary \ref{cor-g(f,f')-lat-obstruc}, although $G(F,F')$ cannot be a lattice in a product, it happens that there exist non-discrete groups $G_1,G_2$ such that $G(F,F')$ embeds as a discrete subgroup of $G_1 \times G_2$ with injective and dense projection to each factor. For instance if $F_1,F_2$ are permutation groups such that $F \lneq F_i \leq F'$ and we set $G_i = G(F_i,F')$, then the diagonal embedding of $G(F,F')$ in $G_1 \times G_2$ has this property as soon as $F_1 \cap F_2 = F$. See \cite[Lem.\ 3.4 and \S 7.1]{LB-ae}.
\end{rmq}

\bibliographystyle{abbrv}
\bibliography{biburs}

\end{document}